\documentclass[smallextended,envcountsect,envcountsame,numbook,nospthms]{svjour3} 
\smartqed 

\makeatletter
\def\cl@chapter{}
\makeatother
\usepackage[utf8]{inputenc}
\usepackage{lmodern}
\usepackage[pdfpagelabels,colorlinks]{hyperref}
\usepackage{mathtools,amsfonts,amssymb,amscd,dsfont}
\usepackage[misc]{ifsym}
\usepackage[shortlabels]{enumitem}
\usepackage{lmodern}
\usepackage{amssymb}
\usepackage{subcaption}
\usepackage{cite}
\usepackage{mathrsfs}
\usepackage[separate-uncertainty = true]{siunitx}
\usepackage{booktabs}
\usepackage[pdftex,dvipsnames]{xcolor}
\usepackage[colorinlistoftodos,prependcaption,textsize=tiny]{todonotes}
\reversemarginpar 
\graphicspath{{./figures/}}

\journalname{Mathematical Programming}

\usepackage{amsthm}
\newtheorem{theorem}{Theorem}[section]
\newtheorem{lemma}[theorem]{Lemma}
\newtheorem{corollary}[theorem]{Corollary}
\newtheorem{proposition}[theorem]{Proposition}
\theoremstyle{definition}
\newtheorem{definition}[theorem]{Definition}

\theoremstyle{remark}

\usepackage[capitalize,nameinlink]{cleveref}
\usepackage{autonum}

\Crefname{lemma}{Lemma}{Lemmas}
\crefname{proposition}{Proposition}{Propositions}
\crefname{definition}{Definition}{Definitions}
\Crefname{proposition}{Proposition}{Propositions}
\Crefname{corollary}{Corollary}{Corollaries}

\newcommand{\NN}{\ensuremath{\mathds N}}
\newcommand{\RR}{\ensuremath{\mathds R}}

\def\re{\mathds R}
\def\na{\mathds N}

\def\DD{\mathbf D}
\def\KK{\mathbf K}

\def\vz{\mathbf z}

\def\lV{\left\lVert }
\def\rV{\right\rVert }

\def\Eset{\mathcal{E}}
\def\Sset{\mathcal{S}}

\DeclareMathOperator{\Id}{Id}
\DeclareMathOperator{\aff}{aff}
\DeclareMathOperator{\dist}{dist}
\DeclareMathOperator{\dom}{dom}
\DeclareMathOperator{\inte}{int}

\DeclareMathOperator{\Fix}{Fix}
\DeclareMathOperator{\diag}{diag}
\newcommand{\SOC}{\mathcal{C}}

\newcommand{\norm}[1]{\left\lVert#1\right\rVert}
\newcommand{\scal}[2]{\left\langle{#1},{#2}  \right\rangle}
\DeclareMathOperator{\circum}{circ}
\newcommand{\MAP}{\ensuremath{\mathrm{MAP}}}

\newcommand{\CRM}{\ensuremath{\mathrm{CRM}}}

\newcommand{\SPM}{\ensuremath{\mathrm{SPM}}}
\newcommand{\PCRM}{\ensuremath{\mathrm{pCRM}}}
\newcommand{\Cent}{\ensuremath{\mathrm{C}}}
\newcommand{\cCRM}{\ensuremath{\mathrm{cCRM}}}
\newcommand{\pCRM}{\ensuremath{\mathrm{pCRM}}}
\newcommand{\pCRMOp}{\ensuremath{\mathscr{C}}}
\newcommand{\CRMprod}{\ensuremath{\mathrm{CRMprod}}}

\usepackage[shortlabels]{enumitem}
\newlist{lista}{enumerate}{1}
\setlist[lista]{label=\alph*., nosep,leftmargin=*,align=right}

\newlist{listi}{enumerate}{1}
\setlist[listi]{label={\upshape(\roman*\upshape)},leftmargin=*,align=right, widest=iii,nosep, format=\bf}

\allowdisplaybreaks

\crefformat{equation}{\textup{#2(#1)#3}}
\crefrangeformat{equation}{\textup{#3(#1)#4--#5(#2)#6}}
\crefmultiformat{equation}{\textup{#2(#1)#3}}{ and \textup{#2(#1)#3}}
{, \textup{#2(#1)#3}}{, and \textup{#2(#1)#3}}
\crefrangemultiformat{equation}{\textup{#3(#1)#4--#5(#2)#6}}%
{ and \textup{#3(#1)#4--#5(#2)#6}}{, \textup{#3(#1)#4--#5(#2)#6}}{, and \textup{#3(#1)#4--#5(#2)#6}}

\Crefformat{equation}{#2Equation~\textup{(#1)}#3}
\Crefrangeformat{equation}{Equations~\textup{#3(#1)#4--#5(#2)#6}}
\Crefmultiformat{equation}{Equations~\textup{#2(#1)#3}}{ and \textup{#2(#1)#3}}
{, \textup{#2(#1)#3}}{, and \textup{#2(#1)#3}}
\Crefrangemultiformat{equation}{Equations~\textup{#3(#1)#4--#5(#2)#6}}%
{ and \textup{#3(#1)#4--#5(#2)#6}}{, \textup{#3(#1)#4--#5(#2)#6}}{, and \textup{#3(#1)#4--#5(#2)#6}}

\crefdefaultlabelformat{#2\textup{#1}#3}

\usepackage[left,mathlines]{lineno}
\usepackage{etoolbox} 
\newcommand*\linenomathpatchAMS[1]{%
  \expandafter\pretocmd\csname #1\endcsname {\linenomathAMS}{}{}%
  \expandafter\pretocmd\csname #1*\endcsname{\linenomathAMS}{}{}%
  \expandafter\apptocmd\csname end#1\endcsname {\endlinenomath}{}{}%
  \expandafter\apptocmd\csname end#1*\endcsname{\endlinenomath}{}{}%
}

\expandafter\ifx\linenomath\linenomathWithnumbers
  \let\linenomathAMS\linenomathWithnumbers
  \patchcmd\linenomathAMS{\advance\postdisplaypenalty\linenopenalty}{}{}{}
\else
  \let\linenomathAMS\linenomathNonumbers
\fi

\linenomathpatchAMS{gather}
\linenomathpatchAMS{multline}
\linenomathpatchAMS{align}
\linenomathpatchAMS{alignat}
\linenomathpatchAMS{flalign}

\begin{document}

\title{On the centralization of the circumcentered-reflection method~\thanks{\textbf{RB} was partially supported by \emph{Conselho Nacional de Desenvolvimento Cient\'ifico e Tecnol\'ogico} (CNPq), Grants 304392/2018-9 and 429915/2018-7 and \emph{Fundação de Amparo à Pesquisa do Estado do Rio de Janeiro} (FAPERJ), Grant  E-26/201.345/2021;  \textbf{YBC} was partially supported by the \emph{National Science Foundation} (NSF), Grant DMS-2307328, and by an internal grant from NIU. \textbf{LRS} was partially supported by CNPq,  Grant 113190/2022-0.}}

\author{Roger Behling  \and Yunier Bello-Cruz  \and 
Alfredo N. Iusem\and Luiz-Rafael Santos 
}

\institute{
  Roger Behling   \Letter  \at School of Applied Mathematics, Funda\c{c}\~ao Get\'ulio Vargas \\ 
Rio de Janeiro, RJ -- 22250-900, Brazil. \email{rogerbehling@gmail.com}
    \and 
Yunier Bello-Cruz \at Department of Mathematical Sciences, Northern Illinois University. \\ 
 DeKalb, IL -- 60115-2828, USA. \email{yunierbello@niu.edu}
\and 
Alfredo N. Iusem
 \at School of Applied Mathematics, Funda\c{c}\~ao Get\'ulio Vargas \\ 
Rio de Janeiro, RJ -- 22250-900, Brazil. \email{iusp@impa.br} 
 \and
 Luiz-Rafael Santos \at Department of Mathematics, Federal University of Santa Catarina. \\ 
Blumenau, SC -- 89065-300, Brazil. \email{l.r.santos@ufsc.br}
}
\date{}

\titlerunning{On the centralization of the circumcentered-reflection method}
\authorrunning{Behling, Bello-Cruz, Iusem and Santos}
\maketitle
\vspace*{-10mm}

\begin{abstract}
This paper is devoted to deriving the first circumcenter iteration scheme that does not employ a product space reformulation for finding a point in the intersection of two closed convex sets. We introduce a so-called centralized version of the circumcentered-reflection method (CRM). Developed with the aim of accelerating classical projection algorithms, CRM is successful for tracking a common point of a finite number of affine sets. In the case of general convex sets, CRM was shown to possibly diverge if Pierra’s product space reformulation is not used. In this work, we prove that there exists an easily reachable region consisting of what we refer to as centralized points, where pure circumcenter steps possess properties yielding convergence. The resulting algorithm is called centralized CRM (cCRM). In addition to having global convergence, cCRM converges linearly under an error bound condition, and superlinearly if the two target sets are so that their intersection have nonempty interior and their boundaries are locally differentiable manifolds. We also run numerical experiments with successful results. 



\end{abstract}

\section{Introduction}\label{s0}

In this work we introduce a new tool for solving the following convex feasibility problem (CFP):
\[\label{eq:CFP}
  \text{ find } z \in X\cap Y,  
\]
where $X,Y\subset \re^n$ are two given closed convex sets with nonempty intersection.

The circumcentered-reflection method (CRM) was presented in 2018 in \cite{Behling:2018} 
as an acceleration technique for classical projection methods. Since then, a quite robust literature related to CRM has been developed~\cite{Araujo:2022,  Arefidamghani:2023, Bauschke:2018, Bauschke:2020, Bauschke:2021, Bauschke:2021b, Bauschke:2021d, Behling:2018a, Behling:2020, Behling:2021b, Behling:2023, Dizon:2022, Dizon:2022a, Lindstrom:2021, Lindstrom:2022,  Ouyang:2018, Ouyang:2021a, Ouyang:2023, Ouyang:2022a, Ouyang:2022b}. 
 If one is at a point $z^{k}\in \re^{n}$, the original CRM  for problem \cref{eq:CFP} moves to the iterate 
\[\label{eq:CRMiterate}
  z^{k+1}_{\CRM}=T_{\CRM}(z^k)\coloneqq \circum(z^{k},R_{X}(z^{k}),R_{Y}R_{X}(z^{k})), 
  \] 
  where $R_X,R_Y:\re^n\to\re^n$ are the orthogonal reflectors through $X,Y$, defined as $R_{X}\coloneqq 2P_X - \Id,R_{Y}\coloneqq 2P_{Y} - \Id$, and $P_{X},P_{Y}:\re^n\to\re^n$ are the orthogonal projections onto $X,Y$, respectively. The Euclidean circumcenter $\circum(z,v,w)$ is the point equidistant to the vertices $z,v,w\in \re^n$ lying on the affine subspace determined by the correspondent triangle  (see \cite[Eq. (2)]{Behling:2018a}). Formally, we have the following definition.
  \begin{definition}[Circumcenter]\label[definition]{def:circum}
  Let $z,v,w\in\re^n$ be given. The \emph{circumcenter} $\circum(z,v,w)\in\re^n$ is a point satisfying
  \begin{enumerate}[(i), format=\bf,leftmargin=*,align=right, widest=ii]
    \item   $\norm{\circum(z,v,w) - x}=\norm{\circum(z,v,w) - y}=\norm{\circum(z,v,w) - z}$ and,
    \item  $\circum(z,v,w)\in \aff\{z,v,w\}\coloneqq \{u\in\re^n \mid u=z+\alpha (v-z)+\beta (w-z),\;\alpha,\beta\in\re\}$.
  \end{enumerate}
  \end{definition}
  
  The point $\circum(z,v,w)$ is well and uniquely defined if the cardinality of the set $\{z,v,w\}$ is one or two. In 
  the case in which the three points are all distinct,  $\circum(z,v,w)$ is well and uniquely defined  only if $x$, $y$ and $z$ 
  are not collinear \cite{Bauschke:2018}.   
  Iteration \cref{eq:CRMiterate}  is well-defined and leads to convergence when $X$ and $Y$ are affine \cite{Behling:2018a}. This is also the case for multi-set affine intersection~\cite{Behling:2020,Behling:2021b} or when the reflectors are substituted by isometries~\cite{Bauschke:2020}. 

  CRM   first aimed to speed up the Douglas-Rachford method (DRM)~\cite{Douglas:1956,Lindstrom:2021,Bauschke:2014b} (also known as the averaged alternating reflections' method). 
Later, in~\cite{Behling:2020}, CRM was connected to the famous method of alternating projections (MAP)~\cite{Bauschke:1993}  whose iteration employs a composition of projections as follows
\[\label{eq:MAP}
z^{k+1}_{\MAP} = T_{\MAP}(z^k) \coloneqq P_{Y}P_{X}(z^{k}).
\]

In \cite{Behling:2018}, Behling, Bello-Cruz and Santos pointed out that iteration \cref{eq:CRMiterate} could fail to
be well-defined or to approach the target set. Later on, Arag{\'o}n~Artacho, Campoy and Tam~\cite[Figure 10]{AragonArtacho:2020} chose an initial point in this very example for which the correspondent CRM sequence actually diverges. Fortunately, 
this was overcome in \cite{Behling:2021b} by considering Pierra's product space reformulation~\cite{Pierra:1984}. 
Pierra stated that problem \cref{eq:CFP} is univocally related to the problem of finding a common point to the diagonal 
subspace $\DD\coloneqq \{(z,z) \mid z \in \re^{n}\}$ and the Cartesian convex set $\KK=X\times Y$. 
In \cite{Behling:2021b} it was shown that a sequence of circumcenters, with initial point in $\DD$ and that iterates as 
 \[\label{eq:CRM-Pierraiterate}
  \vz^{k+1}_{\CRMprod}\coloneqq \circum(\vz^{k},R_{\KK}(\vz^{k}),R_{\DD}R_{\KK}(\vz^{k})),
  \]
converges to a point $\vz^{*} = (z^{*},z^{*})$, where $z^{*} \in X\cap Y$, that is, $z^{*}$ is a solution of problem \cref{eq:CFP}. 
Such a result is indeed derived in \cite{Behling:2021b} for {the case of} a finite number of convex sets.

In this work, we prove that a CRM step based on parallel reflections leads to convergence as long as the iterates stay in an appropriate region. 
We will easily reach this region and get a very fast CRM projection-type method for solving problem \cref{eq:CFP}. 
In our study, the following parallel CRM (pCRM) iteration will be considered
\[\label{eq:pCRM-iteration}
z^{k+1}_{\pCRM}=\pCRMOp(z^k) \coloneqq \circum(z^{k},R_{X}(z^{k}),R_{Y}(z^{k})).\]
We show that this iteration provides adequate steps for solving problem \cref{eq:CFP} 
if the angle between the vectors $R_{X}(z^{k}) - z^{k}$ and $R_{Y}(z^{k}) - z^{k}$ is obtuse (or right). 
When $z^{k}$ satisfies this property, we say that it is \emph{centralized}. Achieving such property is possible, for instance, by taking an appropriate projection procedure. Roughly speaking, we show {that} a MAP step taken from any given $z\in \re^{n}$ 
helps to provide a point with the desired feature. In fact,  
\[\label{eq:CentralizationStep}
  z_{\Cent}\coloneqq \tfrac{1}{2} \left( z_{\MAP} + P_{X}(z_{\MAP}) \right),
\]
is a centralized point, that is, it satisfies 
$\scal{R_{X}( z_{\Cent}) -  z_{\Cent}}{R_{Y}( z_{\Cent}) -  z_{\Cent}}\leq 0$, 
where $\scal{\cdot}{\cdot}$ stands for the Euclidean inner product and $z_{\MAP}\coloneqq P_YP_X(z)$.

A pCRM step \cref{eq:pCRM-iteration} represents an acceleration of the Simultaneous Projections Method (SPM), also called Cimmino's method  \cite{Cimmino:1938}, given by
\[\label{eq:Cimmino}
z^{k+1}_{\SPM} = T_{\SPM}(z^k)\coloneqq \tfrac{1}{2} \left( P_{X}(z^{k}) + P_{Y}(z^{k}) \right),
\]
known to converge to a point in $X\cap Y$ whenever 
$X\cap Y\ne\emptyset$. We mention, parenthetically, that this method is devised for several convex sets with different weights in the average of the projections; iteration \cref{eq:Cimmino} corresponds to the case of two sets with equal weights.

  We note that our centralization procedure \cref{eq:CentralizationStep} comes from the composition of the simultaneous projection operator $T_{\SPM}$ and the alternating projection operator $T_{\MAP}$. Indeed, for any $z\in \re^n$, 
  \begin{align}z_\Cent&=\tfrac{1}{2} \left(z_{\MAP} +  P_{X}(z_{\MAP})\right)\\ &= \tfrac{1}{2} \left(P_{Y}P_{X}(z) +  P_{X}(P_{Y}P_{X}(z))\right)\\&= \tfrac{1}{2} \left(P_{Y}(P_{Y}P_{X}(z)) +  P_{X}(P_{Y}P_{X}(z))\right)\\&=T_{\SPM}(T_{\MAP}(z)).\end{align} 

That said, we can now formulate the centralized circumcentered-reflection method (\cCRM). It iterates by composing MAP, SPM and pCRM given in \cref{eq:pCRM-iteration}, that is, 
  \[\label{eq:cCRM-iteration}
z^{k+1}_{\cCRM}=\pCRMOp(T_{\SPM}(T_{\MAP}(z^{k}))).\]

The goal of our paper is to study cCRM. We will prove  in \Cref{thm:pCRM-centralized} that, for any starting point $z^0\in \re^n$, the sequence generated by iteration \cref{eq:cCRM-iteration} converges to a point in $X\cap Y$. If an error bound condition holds for problem \cref{eq:CFP}, we show in \Cref{t1} that cCRM converges linearly, and we derive an upper bound for its asymptotic error constant. Finally,  \Cref{theorem:superlinearConv} states that, in the case where $X\cap Y$ has nonempty interior and the boundaries of $X$ and $Y$ are locally differentiable manifolds, cCRM actually converges superlinearly.   

\begin{figure*}[!ht]
  \centering 
    \begin{subfigure}{0.48\textwidth}
      \includegraphics[width = \textwidth]{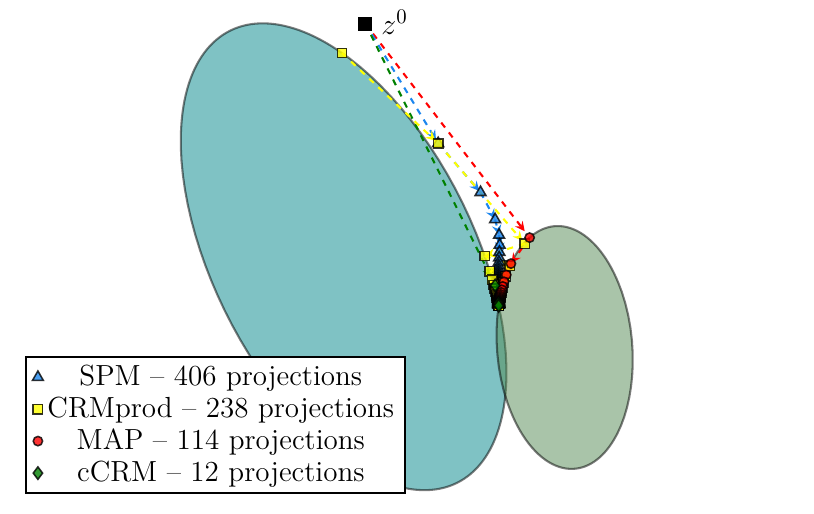}
      \caption{Nonempty interior of intersection.}   \label{fig:TwoEllipsoids-nonemptyinterior}
    \end{subfigure} 
    \begin{subfigure}{0.48\textwidth}
      \includegraphics[width = \textwidth]{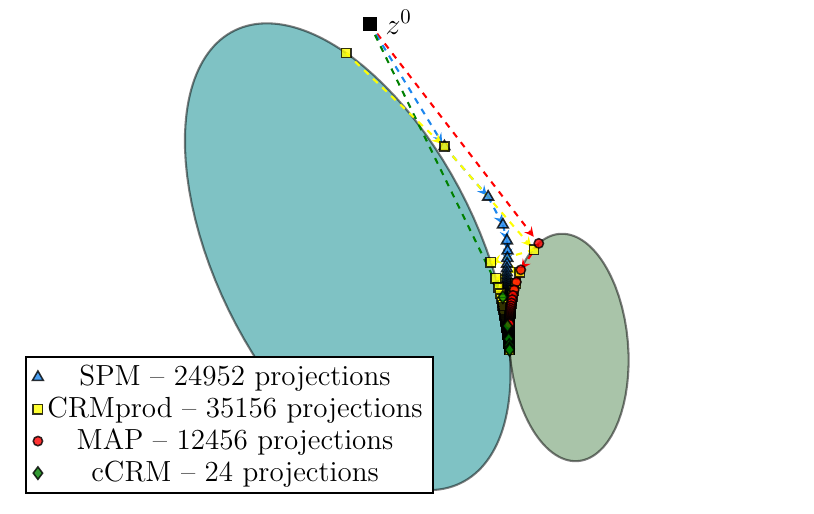}
      \caption{Empty interior intersection.} \label{fig:TwoEllipsoids-emptyinterior}
    \end{subfigure}
    \caption{SPM, CRMprod, MAP, and cCRM paths of iterates on two ellipsoids intersection.}
    \label{fig:TwoEllipsoids}
  \end{figure*}

We present in \Cref{fig:TwoEllipsoids} two instances which illustrate this better performance of cCRM when compared with SPM, CRMprod, and MAP. Note that  SPM, CRMprod, and MAP need to compute two projections (one for each set $X$ and $Y$) at each iterate, while cCRM needs four projections. Thus, we list in the pictures the number of projections that each method takes to achieve convergence, and display correspondent paths towards a solution. We can envision in this example the superiority of cCRM, even when the intersection between the target sets has empty interior; see \Cref{fig:TwoEllipsoids-emptyinterior}. In this numerical example, convergence is understood to occur when the distance of the iterate to the intersection of $X$ and $Y$ is proportional to a tolerance of order $\varepsilon \coloneqq \num{e-4}$. We note that in  \Cref {fig:TwoEllipsoids-nonemptyinterior} the boundaries of the sets $X$ and $Y$  are locally differentiable manifolds and the interior of their intersection is nonempty; that said, the picture displays the superlinear convergence of cCRM, proved later.

The paper is organized as follows. In \cref{s2}, we prove that cCRM 
converges globally to a solution of problem \cref{eq:CFP}. We begin \Cref{s4} with a discussion on error bound conditions;  then, assuming that \cref{eq:CFP} satisfies an error bound condition, we proceed deriving linear convergence of cCRM and an upper bound for the linear rate; finally, under additional hypotheses, we prove that cCRM converges superlinearly.  \Cref{s6} exhibits numerical experiments showing 
cCRM outperforming CRMprod and MAP.  \Cref{s7} presents concluding remarks. 


\section{Convergence of cCRM}\label{s2}

We start this section with the definition of a centralized point.

\begin{definition}[Centralized point]\label[definition]{d1}
Given two closed and convex sets $X,Y\subset\re^n$ a point $z\in\re^n$ is said to be \emph{ centralized with respect to $X,Y$} if 
\[\label{Eq:reflections-central-ineq}
\langle R_X(z)-z,R_Y(z) -z\rangle\le 0.
\]
\end{definition}

Note that we can get an equivalent definition to the one above if we replace the reflections by projections in \cref{Eq:reflections-central-ineq}. Indeed, this is true because
\begin{align}\scal{R_{X}( z_{\Cent}) -  z_{\Cent}}{R_{Y}( z_{\Cent}) -  z_{\Cent}}& =\scal{2P_{X}( z_{\Cent}) - 2z_{\Cent}}{2P_{Y}( z_{\Cent})- 2 z_{\Cent} }\\& = 4 \scal{P_{X}( z_{\Cent}) - z_{\Cent}}{P_{Y}( z_{\Cent})-  z_{\Cent} }. \end{align}
More than that, if $z$ is in $X$ or $Y$, then it is centralized. However, points in $X\cup Y$ are not the most suitable for our algorithmic framework. Therefore, our main interest will be on centralized points that are neither in $X$ nor in $Y$. Those points will be referred to as \emph{strictly centralized}. Note that the definition of strictly centralized point does not necessarily imply that inequality \cref{Eq:reflections-central-ineq} holds strictly.

We will present a series of lemmas  aiming to guarantee that the pCRM iteration  \cref{eq:pCRM-iteration} computed from a strictly centralized point moves towards the solution set of problem \cref{eq:CFP}. However, a pCRM step taken from a non-centralized point may push the next iterate away from the solution set. This behavior is depicted in \Cref{fig:zC_pCRM}. Note in \Cref{fig:zC_pCRMa} that going from $z$ to $z_{\Cent}$ is already better than moving from $z$ to $z_{\pCRM}$. Furthermore, we get even closer to the solution set with $z_{\cCRM}$.
To make this visible, we present part of  \Cref{fig:zC_pCRMa} zoomed in \Cref{fig:zC_pCRMb}. In this way, we illustrate the benefit of the combination of centralization and computation of parallel circumcenter.

\begin{figure*}[!ht]
  \centering 
    \begin{subfigure}{0.45\textwidth}
      \includegraphics[width = \textwidth]{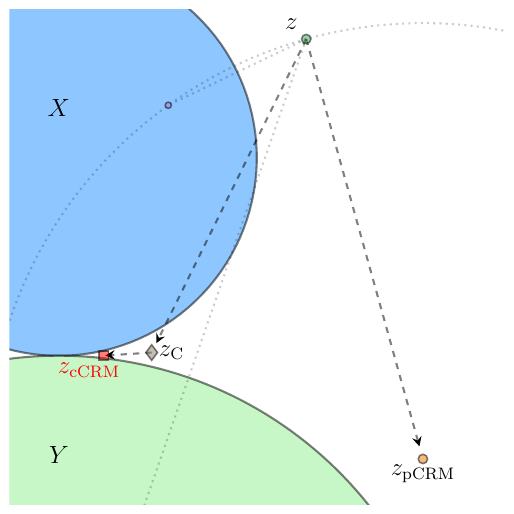}
        \caption{The role of the centralization procedure}   \label{fig:zC_pCRMa}
      \end{subfigure} \qquad 
      \begin{subfigure}{0.45\textwidth}
        \includegraphics[width = \textwidth]{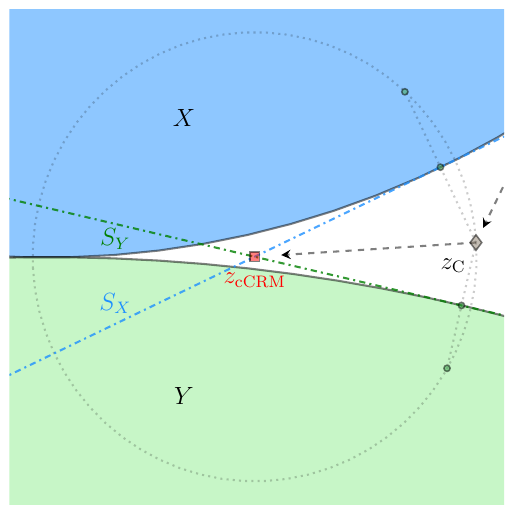}
        \caption{Zoom of the computation $z_{\cCRM}$ from $z_{\Cent}$} \label{fig:zC_pCRMb}
      \end{subfigure}
      \caption{Characterization of centralized circumcenters.}   \label{fig:zC_pCRM}
    \end{figure*}

The next lemma shows that the composition of SPM and MAP, mentioned in the introduction, provides either a strictly centralized point or a solution of problem \cref{eq:CFP}.

\begin{lemma}[Centralization procedure]\label[lemma]{Lemma:CentralProcedure}
Let $X,Y\subset \re^n$ be two closed convex sets with nonempty intersection. For any $z\in \re^n$ set $z_{\MAP}= P_YP_X(z)$. Then, $z_{\Cent}$ from \cref{eq:CentralizationStep}, \emph{i.e.}, $z_{\Cent}\coloneqq \tfrac{1}{2} \left( z_{\MAP} + P_{X}(z_{\MAP}) \right)$,  is centralized. Moreover, $z_{\Cent}$ is either strictly centralized or it belongs to $X\cap Y$. 

\end{lemma}
\begin{proof}
Observe that $P_X(z_{\Cent})=P_X(z_{\MAP})$ and so $z_{\Cent} = \tfrac{1}{2} \left( z_{\MAP} + P_X(z_{\Cent}) \right)$. Hence,
\begin{align}\scal{P_{X}( z_{\Cent}) - z_{\Cent}}{P_{Y}( z_{\Cent})-  z_{\Cent} }&=\scal{2 z_{\Cent} - z_{\MAP} - z_{\Cent}}{P_{Y}( z_{\Cent})-  z_{\Cent} }\\
  &=\scal{z_{\Cent} -  z_{\MAP} }{P_{Y}( z_{\Cent})-  z_{\Cent} }\\
  &= \scal{P_{Y}( z_{\Cent}) -  z_{\MAP} }{P_{Y}( z_{\Cent})-  z_{\Cent} } + \scal{z_{\Cent} -  P_{Y}( z_{\Cent})}{P_{Y}( z_{\Cent})-  z_{\Cent} }\\&= \scal{z_{\Cent} -P_{Y}( z_{\Cent})}{  z_{\MAP} - P_{Y}( z_{\Cent})  } - \|z_{\Cent} -  P_{Y}( z_{\Cent})\|^2\\&\leq \scal{z_{\Cent} -P_{Y}( z_{\Cent})}{  z_{\MAP} - P_{Y}( z_{\Cent})  }\leq 0, \end{align}
where the last inequality follows from the characterization of projections, since $z_{\MAP}\in Y$. This proves that $z_{\Cent}$ is centralized. 

Assume that $z_{\Cent}\in X$. Since \[z_{\Cent}=\tfrac{1}{2} \left( z_{\MAP} + P_{X}(z_{\MAP}) \right)=\tfrac{1}{2} \left( z_{\MAP} + P_{X}(z_{\Cent}) \right)=\tfrac{1}{2} \left( z_{\MAP} + z_{\Cent} \right),\] we have $z_{\Cent}=z_{\MAP}$. Then, $z_{\Cent}\in Y$ because $z_{\MAP}=P_{Y}P_{X}(z)$. Thus,  $z_{\Cent}\in X\cap Y$.

 Now, if $z_{\Cent}\in Y$, note that
\begin{align}\scal{P_X(z)-z_{\MAP}}{P_X(z_{\MAP})-z_{\MAP}} & = 2 \scal{P_X(z)-z_{\MAP}}{z_{\Cent}-z_{\MAP}}\\&=2 \scal{P_X(z)-P_Y(P_X(z))}{z_{\Cent}-P_Y(P_X(z))}\\&\leq 0\label{Eq:projection-ineq1},
\end{align}
where the first and the second equalities follow from the definitions of $z_{\Cent}$ and $z_{\MAP}$, respectively.  The inequality is due to the characterization of the projection of $P_X(z)$ onto $Y$. On the other hand, the characterization of the projection of $z_{\MAP}$ onto $X$ gives us
\begin{align}\scal{z_{\MAP}-P_X(z_{\MAP})}{P_X(z)-P_X(z_{\MAP})} & \leq 0,
\end{align} or equivalently, 
\begin{align}\scal{P_X(z_{\MAP})-P_X(z)}{P_X(z_{\MAP})-z_{\MAP}} & \leq 0\label{Eq:projection-ineq2}.
\end{align}
Summing up \cref{Eq:projection-ineq1,Eq:projection-ineq2}, we get $\|P_X(z_{\MAP})-z_{\MAP}\|^2\le 0$, where $\norm{\cdot}$ stands for the norm induced by the Euclidean inner product. So, $z_{\MAP}=P_X(z_{\MAP})$ and hence, $z_{\MAP}\in X$  and $z_{\Cent}=z_{\MAP}$. Therefore,  $z_{\Cent}\in X\cap Y$.
\end{proof}

Next, we are going to prove that the parallel circumcenter at a centralized point  $z\in \re^n$ is actually the projection of $z$ onto the intersection of two suitable halfspaces defined by the supporting hyperplanes to $X$ and $Y$ passing through $P_{X}(z)$ and $P_{Y}(z)$, respectively; see \Cref{fig:zC_pCRMb}.

\begin{lemma}[Characterization of centralized circumcenters]\label[lemma]{lemma:ZcCharacterization}
  Let $X,Y\subset \re^n$ be two given closed convex sets with nonempty intersection. Assume that  $z\in \re^n$ is a centralized point with respect to $X,Y$. Then, the parallel circumcenter at $z$, 
  \[
    \pCRMOp(z)\coloneqq\circum\{z, R_{X}(z), R_{Y}(z)\},
  \]
coincides with \(P_{S_{X}^{z}\cap S_{Y}^{z}}(z),\)
where 
\[
  S_{X}^{z}  \coloneqq\{w\in \re^n \mid \scal{w-P_X(z)}{z - P_X(z)}\leq 0\},
 \]
 and
 \[
  S_{Y}^{z} \coloneqq\{w\in \re^n \mid \scal{w-P_Y(z)}{z - P_Y(z)}\leq 0\}.
 \]
  
\end{lemma}

\begin{proof}
We start proving the statement for the case in which $z\in \re^n$ lies in one of the sets. Assume, without loss of generality, that $z\in X$, \emph{i.e.}, $z = P_{X}(z) = R_{X}(z)$.  In this case,  $S_{X}^{z} = \re^n$ and $ \pCRMOp(z) = \circum\{z, z, R_{Y}(z)\} = \tfrac{1}{2}\left( z+ R_{Y}(z)\right) = P_{Y}(z)$. Note that $P_{Y}(z) $ is precisely  $P_{S_{Y}^{z}}(z) = P_{\re^n\cap S_{Y}^{z}}(z) = P_{S_{X}^{z}\cap S_{Y}^{z}}(z)  $. 

Now, assume that $z$ is neither in $X$ nor in $Y$, \emph{i.e.}, $z$ is strictly centralized. Thus, both $ S_{X}^{z}$ and $ S_{Y}^{z}$ are actual  half-spaces, because $z - P_X(z)\neq 0$ and $z - P_Y(z)\neq 0$, and  also $z$ is neither in $S_{X}^{z}$ nor $S_{Y}^{z}$.   First, we are going to establish that $z_{\PCRM}$ is the projection of $z$ onto the intersection of the hyperplanes 
\[
  H_{X}^{z}  \coloneqq\{w\in \re^n \mid \scal{w-P_X(z)}{z - P_X(z)}= 0\},
 \]
 and
 \[
  H^{z}_{Y} \coloneqq\{w\in \re^n \mid \scal{w-P_Y(z)}{z - P_Y(z)}= 0\},
 \]
the boundaries of $S^{z}_{X}$ and $S^{z}_{Y}$, respectively. We have $H^{z}_X\cap H^{z}_Y\neq \emptyset$, otherwise the hyperplanes would be parallel and $z$ would be a convex combination of $ P_X(z)$ and $ P_Y(z)$, because $z$ is strictly centralized. In this case,  the half-spaces $ S^{z}_{X}$ and $ S^{z}_{Y}$ would have empty intersection, a contradiction with the facts that $X \cap Y \neq \emptyset$ and $X\subset S^{z}_{X}$ and $Y\subset S^{z}_{Y}$. Then, $ z$, $R_{X}(z)$ and  $R_{Y}(z)$ are not collinear and the circumcenter $\pCRMOp(z)$ is well-defined.  
Moreover, $R_{X}(z)=R_{H^{z}_{X}}(z)$ and  $R_{Y}(z)=R_{H^{z}_{Y}}(z)$. Hence, $\pCRMOp(z)=\circum\{z, R_{H^{z}_X}(z), R_{H^{z}_Y}(z)\}$. Note further that by denoting $\hat z=R_{H^{z}_X}(z)$, we have $z=R_{H^{z}_X}(\hat z)$ and $R_{H^{z}_Y}(z)=R_{H^{z}_Y}R_{H^{z}_X}(\hat z)$, so $\pCRMOp(z)=\circum\{\hat z, R_{H^{z}_X}(\hat z), R_{H^{z}_Y}R_{H^{z}_X}(\hat z)\}$. We do this in order to employ \cite[Lemma 3]{Behling:2021b}, which gives us $\pCRMOp(z)=P_{H^{z}_X\cap H^{z}_Y}(\hat z)$. The fact that $R_{H^{z}_X}$ is an isometry and that $H^{z}_X\cap H^{z}_Y$ is an affine subspace  imply that $\pCRMOp(z)=P_{H^{z}_X\cap H^{z}_Y}(z)$. Since $z$ is neither in $S^{z}_{X}$ nor $S^{z}_{Y}$, $P_{S^{z}_{X}\cap S^{z}_{Y}}(z)$ must lie in the boundary of 
$S^{z}_{X} \cap S^{z}_{Y}$, which on the other hand consists of points that are either in $H^{z}_X$ or in $H^{z}_X$. Therefore, three possibilities  for $P_{S^{z}_{X}\cap S^{z}_{Y}}(z)$ arise: $P_{S^{z}_{X}\cap S^{z}_{Y}}(z)=P_{H^{z}_X}(z)$, $P_{S^{z}_{X}\cap S^{z}_{Y}}(z)=P_{H^{z}_Y}(z)$ or $P_{S^{z}_{X}\cap S^{z}_{Y}}(z)=P_{H^{z}_X\cap H^{z}_Y}(z)$. 

Suppose that $P_{S^{z}_{X}\cap S^{z}_{Y}}(z)=P_{H^{z}_X}(z)$. In particular, $P_{H^{z}_X}(z)\in S^{z}_Y$. Bearing in mind that $P_{H^{z}_X}(z) = P_{X}(z)$ we get
\begin{align}0&\geq \scal{P_{X}(z)-P_Y(z)}{z - P_Y(z)}\\&=\scal{P_{X}(z)-z}{z - P_Y(z)}+\scal{z-P_{Y}(z)}{z - P_Y(z)}\\&=-\scal{P_{X}(z)-z}{P_Y(z)-z}+\|z-P_{Y}(z)\|^2\\&> -\scal{P_{X}(z)-z}{P_Y(z)-z},
\end{align} which contradicts the hypothesis that $z$ is centralized. If we assume that $P_{S^{z}_{X}\cap S^{z}_{Y}}(z)=P_{H^{z}_Y}(z)$, then we get a similar contradiction. Thus, $P_{S^{z}_{X}\cap S^{z}_{Y}}(z)=P_{H^{z}_X\cap H^{z}_Y}(z)$, proving the lemma.
\end{proof}



We are going to state now a Fejér-type property regarding a step from $z\in \re^n$ to the centralized point $z_{\Cent}$. This property is quite natural, if one bears in mind that our centralization procedure comes from the composition of SPM and MAP.

\begin{lemma}[Firm quasi-nonexpansiveness of the centralization procedure]\label[lemma]{lemma-quasinonexp-Cent}
 Let $X,Y\subset \re^n$ be two given closed convex sets with nonempty intersection. Then, for any  $z\in \re^n$, the centralized procedure $z_{\Cent}$ given in \cref{eq:CentralizationStep} satisfies 
\[\label{eq:lemmaqnonexpans1}\|z_{\Cent}-s\|^2 \le \|z_{\MAP}-s\|^2 -\tfrac{1}{2}\|z_{\MAP}-P_{X}(z_{\MAP})\|^2, \] 
and
\[\label{eq:lemmaqnonexpans2}\|z_{\Cent}-s\|^2 \le \|z-s\|^2-\tfrac{1}{4}\|z-z_{\Cent}\|^2, \] 
for all $s\in X\cap Y$.
\end{lemma}
\begin{proof}
For any $s\in X\cap Y$, we have
\begin{align}\|z_{\Cent}-s\|^2 &=\|\tfrac{1}{2} \left( z_{\MAP} + P_{X}(z_{\MAP}) \right)-s\|^2
=\|\tfrac{1}{2} ( z_{\MAP} - s) + \tfrac{1}{2}(P_{X}(z_{\MAP})- s) \|^2\\&= \tfrac{1}{2}\|  z_{\MAP} - s\|^2+\tfrac{1}{2}\|P_{X}(z_{\MAP})- s\|^2-\tfrac{1}{4}\|( z_{\MAP} - s) - (P_{X}(z_{\MAP})- s)\|^2\\&\le \tfrac{1}{2}\|  z_{\MAP} - s\|^2+\tfrac{1}{2}\|P_{X}(z_{\MAP})- P_X(s)\|^2
\\&\le \|z_{\MAP}-s\|^2-\tfrac{1}{2}\|z_{\MAP}-P_{X}(z_{\MAP})\|^2\\
&\le  \|z-s\|^2-\tfrac{1}{2}\|z-z_{\MAP}\|^2-\tfrac{1}{2}\|z_{\MAP}-P_{X}(z_{\MAP})\|^2\\
&=  \|z-s\|^2-\tfrac{1}{2}\|z-z_{\MAP}\|^2-2\|\tfrac{1}{2}(z_{\MAP}-P_{X}(z_{\MAP}))\|^2\\
&=  \|z-s\|^2-\tfrac{1}{2}\|z-z_{\MAP}\|^2-2\|z_{\Cent}-z_{\MAP}\|^2\label{Eq-final-use}.
\end{align} In the first equality we use the definition of $z_{\Cent}$ and the second one is obvious. The third equality follows from the identity in \cite[Corollary 2.15]{Bauschke:2017a}. In the first inequality we take into account the fact that $s=P_X(s)$, and the nonnegativity of the last term. The second inequality follows from the firm nonexpasiviness of projections, and we get \cref{eq:lemmaqnonexpans1}. In the last inequality we use \cite[Proposition 4.35(iii)]{Bauschke:2017a}, since the projections are $1/2$-averaged. 
Finally, in the last equality we use again the definition of $z_{\Cent}$.

Note further that \cite[Corollary 2.15]{Bauschke:2017a} gives us
\begin{align}\tfrac{1}{2}\|z-z_{\MAP}\|^2+\tfrac{1}{2}\|z_{\Cent}-z_{\MAP}\|^2&=\|\tfrac{1}{2}(z-z_{\MAP})+\tfrac{1}{2}(z_{\Cent}-z_{\MAP})\|^2+\tfrac{1}{4}\|(z-z_{\MAP})-(z_{\Cent}-z_{\MAP})\|^2\\&\ge\tfrac{1}{4}\|z-z_{\Cent}\|^2,\end{align}
or, equivalently,
\begin{align}\tfrac{1}{2}\|z-z_{\MAP}\|^2&\ge\tfrac{1}{4}\|z-z_{\Cent}\|^2-\tfrac{1}{2}\|z_{\Cent}-z_{\MAP}\|^2.\end{align}
Combining the last inequality with \cref{Eq-final-use}, we get
\begin{align}\|z_{\Cent}-s\|^2 &\le  \|z-s\|^2-\tfrac{1}{4}\|z-z_{\Cent}\|^2+\tfrac{1}{2}\|z_{\Cent}-z_{\MAP}\|^2 -2\|z_{\Cent}-z_{\MAP}\|^2\\&\leq  \|z-s\|^2-\tfrac{1}{4}\|z-z_{\Cent}\|^2.
\end{align}
This proves \cref{eq:lemmaqnonexpans2}, and hence the lemma.
\end{proof}

The next lemma establishes a result similar to the previous one, but now concerning pCRM steps  taken from centralized points.

\begin{lemma}[Firm quasi-nonexpansiveness of circumcenters at centralized points]\label[lemma]{lemma-quasinonexpCRM}
 Let $X,Y\subset \re^n$ be two closed convex sets with nonempty intersection. Assume that  $z\in \re^n$ is a centralized point with respect to $X,Y$. Then, the parallel circumcenter at $z$, namely $\pCRMOp(z)\coloneqq \circum\{z, R_{X}(z), R_{Y}(z)\}$,  satisfies
\[\label{eq:lemmaqnonexpansCRM1}\|\pCRMOp(z)-s\|^2 \le \|z-s\|^2-\|z-\pCRMOp(z)\|^2,
\] for all $s\in X\cap Y$.
\end{lemma}
\begin{proof}  

Consider $S^{z}_{X}$ and  $S^{z}_{Y}$ as in \Cref{lemma:ZcCharacterization}. Now, for any $s\in X\cap Y$, we have 
\begin{align}\|\pCRMOp(z)-s\|^2 &=\|P_{S^{z}_X\cap S^{z}_Y}(z)-s\|^2\\
&=\|P_{S^{z}_X\cap S^{z}_Y}(z)-P_{S^{z}_X\cap S^{z}_Y}(s)\|^2\\&\le \|z-s\|^2-\left \|(z-P_{S^{z}_X\cap S^{z}_Y}(z))-(s-P_{S^{z}_X\cap S^{z}_Y}(s))\right\|^2 \\
&= \|z-s\|^2-\|z-\pCRMOp(z)\|^2,
\end{align}
where the first equality follows from \Cref{lemma:ZcCharacterization}, because $z$ is centralized. The second equality follows from the fact that $s\in X\cap Y\subset S^{z}_X\cap S^{z}_Y$. In the inequality, we invoke the firm nonexpansiveness of projections, and in the last equality, we use the fact that $s=P_{S^{z}_X\cap S^{z}_Y}(s)$, and \Cref{lemma:ZcCharacterization} again.  
\end{proof}

Finally, we can derive firm quasi-nonexpansiveness of full cCRM steps by bonding \cref{lemma-quasinonexp-Cent,lemma-quasinonexpCRM}.

\begin{lemma}[Firm quasi-nonexpansiveness of the centralized circumcentered-reflection operator]\label[lemma]{lemma-quasinonexp-pCRM-combined}
 Let $X,Y\subset \re^n$ be two closed convex sets with nonempty intersection. Let  $z\in \re^n$. Then, the parallel circumcenter at $z_{\Cent}\coloneqq  \tfrac{1}{2} \left( z_{\MAP} + P_{X}(z_{\MAP}) \right)$, namely 
 \[\label{eq:pCRM_operator}
  T(z)\coloneqq \pCRMOp(z_{\Cent})=\circum\{z_{\Cent}, R_{X}(z_{\Cent}), R_{Y}(z_{\Cent})\},\]  satisfies
\[\label{eq:lemmaqnonexpansCRM}\|T(z)-s\|^2 \le\|z-s\|^2-\tfrac{1}{8}\|z-T(z)\|^2,
\] for all $s\in X\cap Y$.
\end{lemma}
\begin{proof} By \Cref{Lemma:CentralProcedure},  $z_{\Cent}$ is centralized. Therefore,  
\Cref{lemma-quasinonexpCRM} can be applied and implies that
\begin{align}\|\pCRMOp(z_{\Cent})-s\|^2 &\le \|z_{\Cent}-s\|^2-\|z_{\Cent}-\pCRMOp(z_{\Cent})\|^2,
\end{align} for all $s\in X\cap Y$. Now, using \Cref{lemma-quasinonexp-Cent} in the previous inequality, we have
\begin{align}\|\pCRMOp(z_{\Cent})-s\|^2 &\le \|z-s\|^2-\tfrac{1}{4}\|z-z_{\Cent}\|^2-\|z_{\Cent}-\pCRMOp(z_{\Cent})\|^2.\label{Eq-final-useb}
\end{align}
Note further that \cite[Corollary 2.15]{Bauschke:2017a} gives us
\begin{align}\tfrac{1}{2}\|z-z_{\Cent}\|^2+\tfrac{1}{2}\|z_{\Cent}-\pCRMOp(z_{\Cent})\|^2&=\|\tfrac{1}{2}(z-z_{\Cent})+\tfrac{1}{2}(z_{\Cent}-\pCRMOp(z_{\Cent}))\|^2+\tfrac{1}{4}\|(z-z_{\Cent})-(z_{\Cent}-\pCRMOp(z_{\Cent}))\|^2\\&\ge \tfrac{1}{4}\|z-\pCRMOp(z_{\Cent})\|^2 ,\end{align}
which can be written as 
\begin{align}\tfrac{1}{4}\|z-z_{\Cent}\|^2&\ge\tfrac{1}{8}\|z-\pCRMOp(z_{\Cent})\|^2-\tfrac{1}{4}\|z_{\Cent}-\pCRMOp(z_{\Cent})\|^2.\end{align}
Combining the last inequality with \cref{Eq-final-useb}, we obtain
\begin{align}\|\pCRMOp(z_{\Cent})-s\|^2 &\le \|z-s\|^2-\tfrac{1}{8}\|z-\pCRMOp(z_{\Cent})\|^2+\tfrac{1}{4}\|z_{\Cent}-\pCRMOp(z_{\Cent})\|^2-\|z_{\Cent}-\pCRMOp(z_{\Cent})\|^2\\&=\|z-s\|^2-\tfrac{1}{8}\|z-\pCRMOp(z_{\Cent})\|^2-\tfrac{3}{4}\|z_{\Cent}-\pCRMOp(z_{\Cent})\|^2\\&\le \|z-s\|^2-\tfrac{1}{8}\|z-\pCRMOp(z_{\Cent})\|^2,
\end{align} 
proving the lemma. 
\end{proof} 

Next,  we show that  $X\cap Y$ is precisely the  set of fixed points of the operator $\pCRMOp: \dom (\pCRMOp) \subset \re^n\to\re^n$, when $X\cap Y\neq \emptyset$.
Here, $\dom (\pCRMOp)$ consists of the points of $\RR^n$ for which 
\cref{eq:pCRM-iteration} is well-defined, i.e.,   $\dom (\pCRMOp) \coloneqq \{z\in \RR^n\mid z, R_{X}(z),\text{ and } R_{Y}(z)\text{ are not collinear}\}$.

\begin{lemma}[Fixed points of the parallel circumcenter operator]\label[lemma]{lemma-FixpCRM}
  Let $X,Y\subset \re^n$ be two given closed convex sets with nonempty intersection. Consider the parallel circumcenter operator   $\pCRMOp: \dom (\pCRMOp) \subset \re^n\to\re^n$ defined as $\pCRMOp(z)\coloneqq \circum\{z, R_{X}(z), R_{Y}(z)\} $. Let  $\Fix \pCRMOp \coloneqq \{z\in \RR^n \mid \pCRMOp(z) = z\}$ be the set of its fixed points. Then,
   \[
     \Fix \pCRMOp = X\cap Y.  
   \]
  \end{lemma}
 \begin{proof}
  If $z\in  X\cap Y$, it is easy to see that $z \in \Fix \pCRMOp$. Now, suppose $z\in  \Fix \pCRMOp$. By \Cref{def:circum} of circumcenter we get that 
  \[
    \norm{z - \pCRMOp(z) } = \norm{R_X(z) -\pCRMOp(z)}  =   \norm{R_Y(z) -\pCRMOp(z)}.  
  \]
  By definition of fixed points, we have $z =  \pCRMOp(z)$, which gives us $ R_X(z) = \pCRMOp(z) = z$ and   $ R_Y(z) = \pCRMOp(z) = z$. Therefore,  because $ \Fix R_X = X$ and $\Fix R_Y = Y$, we get $z\in X\cap Y$, as required. 
 \end{proof}

\begin{lemma}[Fixed points of the centralized circumcentered-reflection operator]\label[lemma]{lemma-FixCirc-Cent}
 Let $X,Y\subset \re^n$ be two closed convex sets with nonempty intersection. Consider the centralized circumcentered-reflection operator $T :\re^n\to\re^n$ at $z\in\re^n$ defined, as in \Cref{lemma-quasinonexp-pCRM-combined}, 
 and the set of fixed points of $T$, namely, $\Fix T \coloneqq \{z\in \re^n\mid T(z) = z\}$. Then, 
  \[
    \Fix T = X\cap Y.  
  \]
 \end{lemma}
 \begin{proof}
If $z\in  X\cap Y$, we clearly get that $z \in \Fix T$. Conversely, suppose $z\in \Fix T$.

Since by \cref{Lemma:CentralProcedure} $z_{\Cent}$ is centralized, \Cref{lemma-quasinonexpCRM} applies to it, that is, for any $s\in X\cap Y$, 
\[\|\pCRMOp(z_{\Cent})-s\|^2 \le \|z_{\Cent}-s\|^2-\|z_{\Cent}-\pCRMOp(z_{\Cent})\|^2.\]
Moreover, as $\pCRMOp(z_{\Cent}) = T(z) = z$, we get 
\begin{align}\|z-s\|^2 &  \le \|z_{\Cent}-s\|^2-\|z_{\Cent}-z\|^2  \\
&  \le \|z-s\|^2-\tfrac{1}{4}\|z-z_{\Cent}\|^2-\|z_{\Cent}-z\|^2,
\end{align}
where in the second inequality we used 
\cref{eq:lemmaqnonexpans2}. Thus,  $\|z_{\Cent}-z\| = 0$, \emph{i.e.}, $z_{\Cent}=z$. In this case,    \[z=T(z) = \circum\{z_{\Cent}, R_{X}(z_{\Cent}), R_{Y}(z_{\Cent})\}= \circum\{z, R_{X}(z), R_{Y}(z) \} = \pCRMOp(z)\] and so $ z\in \Fix \pCRMOp$. Hence, \Cref{lemma-FixpCRM} gives us that $z\in  X\cap Y$. 
\end{proof}

Before arriving  at the main result of our paper,  we recall the notion of Fej\'er monotonicity.

\begin{definition}[Fejér monotonicity]
A sequence $(w^k)_{k\in \na}\subset\re^n$ is \emph{Fej\'er monotone} with respect to a set $M\subset\re^n$ when 
$\lV w^{k+1}-w\rV\le\lV w^k-w\rV$, for all $w\in M$ and for all $k\in \NN$.
\end{definition}

We will now state that if we iterate the cCRM operator $T$, with any choice of initial point in $\re^n$, we end up with a sequence whose limit point exists and belongs to $X\cap Y$. In other words, cCRM solves problem \cref{eq:CFP}. The convergence is derived upon the Fejér monotonicity with respect to $X\cap Y$ of the sequence generated by cCRM, which is directly implied by \Cref{lemma-quasinonexp-pCRM-combined}.

\begin{theorem}[Convergence of cCRM]\label{thm:pCRM-centralized}
  Let $X,Y\subset \re^n$ be two closed convex sets with nonempty intersection. Then, for any starting point   $z^0\in \re^n$, the sequence defined by $z^{k+1} \coloneqq  T(z^k) = \pCRMOp(z_{\Cent}^{k})=\circum\{z_{\Cent}^{k}, R_{X}(z_{\Cent}^{k}), R_{Y}(z_{\Cent}^{k})\}$, where   $z_{\Cent}^{k}\coloneqq\tfrac{1}{2} \left( z_{\MAP}^k + P_{X}(z_{\MAP}^k) \right)$  and $z_{\MAP}^k \coloneqq P_{Y}P_{X}( z^{k})$, is Fej\'er monotone with respect to  $X\cap Y$, and  converges to a point in $X\cap Y$.
\end{theorem}

\begin{proof}
  For all $s\in X\cap Y$ and $k \in \na$,  we get from \Cref{lemma-quasinonexp-pCRM-combined} 
  that
  \[ \frac{1}{8}\|z^{k+1} - z^{k}\|^2\le\|z^{k}-s\|^2-\|z^{k+1}-s\|^2.\label{eq:thm:pCRM-centralized2.2}
\]
and, therefore,  
  \[\label{eq:thm:pCRM-centralized2.1}
  \|z^{k+1}-s\|  \leq \|z^{k}-s\|.
\]
Inequality \cref{eq:thm:pCRM-centralized2.1} provides the Fej\'er monotonicity, with respect to $X\cap Y$, of  sequence $(z^k)_{k\in\na}$, and we get the first claim. Moreover, appealing to~\cite[Proposition 5.4(i)]{Bauschke:2017a},    $(z^k)_{k\in\na}$ is also bounded. To complete the proof, it suffices to show that every cluster point of $(z^k)_{k\in\na}$ belongs to $X\cap Y$, because then the Fejér monotonicity of $(z^k)_{k\in\na}$ implies its convergence to a point in  $X\cap Y$, taking into account \cite[Theorem 5.5]{Bauschke:2017a}.

We proceed to establish the claim. First note that using inequality \cref{eq:thm:pCRM-centralized2.2} and the fact that  sequence $(\|z^{k}-s\|)_{k\in \na}$ converges~\cite[Proposition 5.4(ii)]{Bauschke:2017a}, we conclude that   $z^{k+1}-z^{k}$ converges to $0$, as $k \rightarrow+\infty$. Furthermore,  using \Cref{lemma-quasinonexpCRM} and the nonexpansiveness of the centralized procedure $z_\Cent$ given in \Cref{lemma-quasinonexp-Cent}, we have
\begin{align}
  \|z^{k+1} - z_{\Cent}^k\|^2 
  & = \|\pCRMOp(z_{\Cent}^k) - z_{\Cent}^k\|^2\\ 
  &  \leq   \|z_{\Cent}^k-s\|^2 - 
\|\pCRMOp(z_{\Cent}^k)-s\|^2 \\
& = \|z_{\Cent}^k-s\|^2 - 
\|z^{k+1}-s\|^2  \\ 
& \leq \|z^k-s\|^2 - 
\|z^{k+1}-s\|^2.
\end{align}
This last inequality, again by the convergence of  $(\|z^{k}-s\|)_{k\in \na}$, implies that $z^{k+1}-z_{\Cent}^{k}$ converges to $0$ as $k \rightarrow+\infty$. 

 Now, let $\hat z$ be any cluster point of the sequence $(z^k)_{k\in\na}$ and denote $(z^{i_k})_{k\in\na}$ an associated subsequence convergent  to $\hat z$. Since $z^{i_k}-z^{i_k+1}\to 0$ and $z^{i_k+1}-z_{\Cent}^{i_k}\to 0$, we have   $z^{i_k+1}\to \hat z$ and $z_{\Cent}^{i_k}\to \hat z$. 
 
 We claim that $\hat z\in X\cap Y$. In fact, by the definition of $\pCRMOp$, we have
\[
  \|z^{i_k}_{\Cent}-z^{i_k+1}\| =  \| z^{i_k}_{\Cent} - \pCRMOp(z_{\Cent}^{i_k})\|  = \| R_X(z^{i_k}_{\Cent}) - \pCRMOp(z_{\Cent}^{i_k})\| =\| R_X(z^{i_k}_{\Cent}) - z^{i_k+1}\|.
\] 
Thus, $R_X(z^{i_k}_{\Cent}) - z^{i_k+1}$ converges to $0$. Taking limits as $k\to+\infty$, it follows from the continuity of the reflection onto $X$ that $\hat z=R_X(\hat z)$. Hence, $\hat z\in X$. Since $\| R_X(z^{i_k}_{\Cent}) - z^{i_k+1}\| = \| R_Y(z^{i_k}_{\Cent}) - z^{i_k+1}\|$, we conclude by the same token that $\hat z=R_Y(\hat z)$ and so $\hat z\in Y$, proving the claim, which completes the proof.  
\end{proof}

We have just proven the global convergence of cCRM. In the next section we study the convergence rate of cCRM under an error bound condition.

\section{Convergence order of  cCRM}\label{s4}

The aim of this section is three-fold. We present a discussion on error bound conditions, which are regularity assumptions widely employed in continuous optimization. Under such hypothesis, we derive linear convergence of  cCRM, and we provide an upper bound for its asymptotic constant. Then, under additional mild assumptions, we establish superlinear convergence of cCRM.

\subsection{Error bound condition}\label{s3}

The analysis of the convergence speed of the sequence generated by cCRM is going to be carried out under an additional assumption on the problem. We will assume a fairly standard  local {\it error bound} (EB) condition,  also called  linear regularity~\cite{Bauschke:1993,Bauschke:1996} or subtransversality \cite{Kruger:2018}. 

\begin{definition}[Error bound]\label[definition]{def:EB} 
  Let $X, Y \subset \re^{n}$ be  closed convex sets and assume that $X\cap Y \neq \emptyset$.
  We say that $X$ and $Y$ satisfy a local \emph{error bound} condition if for some point $\bar{z}\in X\cap Y$ there exist a real number  $\omega\in(0,1)$  and a neighborhood $V$ of $\bar z$ such that  
  \[\label{eb}
 \omega \dist(z,X\cap Y) \leq   \max \{\dist(z, X), \dist(z, Y)\},\tag{EB}
  \]
  for all $z\in V$.
  \end{definition}

 The condition given by \cref{def:EB} 
 means that a point in $V$ cannot be too close to both $X$ and $Y$ and at the same time far from $X\cap Y$. More than that, roughly speaking, the constant $\omega$ emulates the sine of the ``angle'' between $X$ and $Y$.  
 This condition will be required in our convergence rate analysis and referred to as Assumption \ref{eb}. 
 

Assumption \ref{eb} is equivalent to asking the existence of a constant $\kappa\in (0,1)$ such that  
\[
\kappa\dist(z,X\cap Y)\le \dist(z,X)\label{eb1fufu}\] for all $z\in Y$ sufficiently close  to $\bar z \in X\cap Y $. This equivalence was discussed in \cite[Section 3.2]{Behling:2021a} and is a consequence of \cite[Lemma 4.1]{Bauschke:1993}. 
The error bound version  \cref{eb1fufu} was used in 
\cite{Arefidamghani:2021}  for establishing linear convergence of the CRM method \cref{eq:CRMiterate}. In that paper, $Y$ is assumed to be an affine manifold, and the whole CRM sequence stays in $Y$.

We relate next, Assumption \ref{eb} with other error bounds found in the literature. If the problem at hand consists of solving $H(z)=0$ with a smooth $H:\re^n\to\re^m$, a classical regularity condition demands that $m=n$ and the Jacobian matrix of $H$ be nonsingular at a solution $z^*$, in which case Newton's method, for instance, is known to enjoy superlinear or quadratic convergence. This condition implies local uniqueness of the solution $z^*$. For problems with $m \ne n$ or
with nonisolated solutions, a less demanding assumption is the notion of {\it calmness} (see \cite{Rockafellar:2004}, Chapter 8, Section F), which requires that 
\begin{equation}\label{fufu}
 \dist(z,S^*) \le \theta  \lV H(z)\rV
\end{equation}
for all $z\in\re^n\setminus S^*$ and some $\theta>0$, where $S^*$ is the solution set, \emph{i.e.}, the set of zeros of $H$.
 Calmness, also called upper-Lipschitz continuity (see \cite{Robinson:1982}), is a classical example of error bound, and it holds in many situations (\emph{e.g.}, when $H$ is affine, by virtue of Hoffman's Lemma, \cite{Hoffman:1952}). It implies that the solution set is locally a Riemannian manifold (see \cite{Behling:2013b}), and it has been used for establishing superlinear convergence of Levenberg-Marquardt methods in \cite{Kanzow:2004}.
 
 We will present next an error bound for systems of inequalities and a result establishing that this error bound holds whenever the system of inequalities satisfies some well known constraint qualifications. This result is a particular instance of a theorem in \cite{Robinson:1976}. 
 
 We recall first the Mangasarian-Fromovitz constraint qualification (MFCQ) for a system of nonlinear inequalities. Let $g_i:\re^n\to\re$, ($i = 1,\ldots, m$) be continuously differentiable and convex functions, and define $S\subset\re^n$ as $S\coloneqq \{z\in\re^n\mid g_i(z)\le 0, \,i=1,\ldots, m\}$. Take $z\in S$ and let $I(z)=\{i\in \{1,\dots, m\}\mid g_i(z)=0\}$.
 MFCQ is said to hold at $z$ if 
 $\sum_{i\in I(z)}\lambda_i\nabla g_i(z)=0$ with $\lambda_i\ge 0$ for all
 $i\in I(z)$ implies that $\lambda_i=0$ for all $i\in I(z)$. Note that if MFCQ holds at $z$ then $\nabla g_i(z)\ne 0$ for all $i$; if $\nabla g_i(z)=0$ then the statement of MFCQ fails if we take  $\lambda_i =1$ and $\lambda_j = 0$, for $j\in I(z)\setminus \{i\}$. 
 
 \begin{proposition}[Mangasarian-Fromovitz implies calmness]\label[proposition]{p4}
 Let $g_i:\re^n\to\re$, $(1\le i\le m)$ be continuously differentiable and convex functions. Fix $\bar z\in S$, with $S$ as above.
 Define $g_i^+:\re^n\to\re$ as
 $g_i^+(z)=\max\{0,g_i(z)\}$ and $g^+:\re^n\to\re^m$ as 
 $g^+(z)=(g_1^+(z),\dots, g_m^+(z))$. If \emph{MFCQ} holds at $\bar z$ then there exists a neighborhood $U$ of $\bar z$ and a constant $\theta>0$ such that
 \begin{equation}\label{f1}
 \dist(z,S)\le \theta\lV g^+(z)\rV,
 \end{equation}
 for all $z\in U$.
 \end{proposition}
 
 \begin{proof} 
 The result is a simplified version of Example 2.92 in \cite{Bonnans:2000} (which also includes equality constraints and an additional parameter),
 which is itself a particular case of Theorem 2.87 in the same reference,
 taken from \cite{Robinson:1976}.
\end{proof}	

When dealing with convex feasibility problems, as in this paper, it seems reasonable to replace the right-hand side of \cref{fufu} by the distance from $z$ to some convex sets, giving rise to \ref{eb}.
Similar error bounds for feasibility problems can be found, for instance, in \cite{Bauschke:1993,Bauschke:1996,Bauschke:1996a,Drusvyatskiy:2015a,Kruger:2018}.
To our knowledge, no extension of  \Cref{p4} to these error bounds has been proved. We proceed to establish such an extension for Assumption \ref{eb}. 

Although we deal in this paper with the intersection of just two convex sets, we will present the result for the more general case of $m$ convex sets $X_1, \dots ,X_m$. Let $X^* \coloneqq \bigcap_{i=1}^mX_i$. For this case, Assumption \ref{eb} becomes the following condition: 
given $\bar z\in X^*$, there exists a neighborhood $V$ of $\bar z$ and a real number $\omega\in(0,1)$ such that
\begin{equation}\label{f2}
  \omega \dist(z,X^*) \le \max_{1\le i\le m}\{\dist(z,X_i)\}
\end{equation} 
for all $z\in V$.
We prove now that condition \cref{f2} holds around $\bar z$ under reasonable regularity assumptions on the convex sets. We assume that $X_i=\{z\in\re^n\mid  g_i(z)\le 0\}$ for some continuously differentiable convex function
$g_i:\re^n\to\re$. This assumption, in principle, entails no loss of generality; we can always take $g_i(z)\coloneqq \lV z-P_{X_i}(z)\rV^2=\dist^2(z,X_i)$, which defines a convex and continuously differentiable function. The second assumption is that MFCQ holds at $\bar z$, and is a rather standard regularity condition;we mention that in most cases the representation of $X_i$ with the above defined function $g_i$ fails to satisfy MFCQ. The result is as follows.

\begin{theorem}[Mangasarian-Fromovitz implies error bound]\label{t2}
If $X_i=\{z\in\re^n:g_i(z)\le 0\}$ for some convex and continuously differentiable $g_i:\re^n\to\re$ and \emph{MFCQ} holds at some point
$\bar z\in X^*\coloneqq \bigcap_{i=1}^m X_i$ then condition \cref{f2} holds around $\bar z$.
\end{theorem}

\begin{proof} 
By \Cref{p4} there exists a neighborhood  $U$ of $\bar z$ such that
\begin{equation}\label{ff2}
 \dist(z,X^*)\le \theta\lV g^+(z)\rV
 \end{equation}
 for all $z\in U$.
 
 Since all the $g_i$'s are continuously differentiable, there exists $\rho >0$ such that the open ball $B(\bar z,\rho)\subset U$ and that 
 $\lV\nabla g_i(z)\rV\le 2\lV\nabla g_i(\bar z)\rV$ for all 
 $i\in\{1, \dots ,m\}$ and all $z\in B(\bar z,\rho)$. Take any $z\in B(\bar z,\rho)$ and let $z_i\coloneqq P_{X_i}(z)$. Expanding $g_i$ around $z_i$, we have
 \begin{equation}\label{f3}
 g_i(z)=g_i(z_i)+ \scal{\nabla g_i(y_i)}{z-z_i}
 \end{equation}
 for some $y_i$ in the segment between $z$ and $z_i$. Since $z_i\in X_i$,
 we have $g_i(z_i)\leq 0$, and it follows from \cref{f3} that
 \begin{equation}\label{f4}
 g_i(z)\le\lV\nabla g_i(y_i)\rV\,\lV z-z_i\rV=
 \lV\nabla g_i(y_i)\rV \dist(z,X_i).
 \end{equation}
 Since $\bar z\in X_i$ for all $i$, the definition  of orthogonal projection implies that $\lV z-z_i\rV\le\lV z-\bar z\rV$, so that
 $z_i\in B(\bar z,\rho)$, and hence $y_i\in B(\bar z,\rho)$, by convexity of the ball. By definition of $\rho$ we have that 
 $\lV\nabla g_i(y_i)\rV\le 2\lV\nabla g_i(\bar z)\rV$. Define now 
 $\sigma_i\coloneqq 2\lV\nabla g_i(\bar z)\rV$; because  MFCQ holds at $\bar z$, we get  $\nabla g_i(\bar z)\neq 0$ and so $\sigma_i > 0$.   Thus, we conclude from \cref{f4} that
 \begin{equation}\label{f5}
 g_i(z)\le\sigma\dist(z,X_i) 
 \end{equation}
 for all $i\in\{1,\dots, m\}$ and all $z\in B(\bar z,\rho)$, with $\sigma\coloneq \max
 \{\sigma_i\mid {1\le i\le m}\}$.
 Now we replace the 2-norm by the $\infty$-norm in \cref{ff2}, obtaining
 \begin{equation}\label{f6}
\dist(z,X^*)\le\theta\lV g^+(z)\rV
\le\theta\sqrt{n}\lV g^+(z)\rV_\infty.
\end{equation}
Note that \cref{f2} holds trivially when $z\in X^*$ because in such a case all distances vanish. Hence, we may assume that there exists $\ell$ such that $g_\ell(z)>0$, and so  
$\lV g^+(z)\rV_\infty=\max_{1\le i\le m}\{g_i(z)\}$. 
Combining \cref{f5} and
\cref{f6} we obtain
$$
\dist(z,X^*)\le\theta\sqrt{n}\max_{1\le i\le m}\{g_i(z)\}\le
\sigma\theta\sqrt{n}\max_{1\le i\le m}\{\dist(z,X_i)\}
$$
for all $z\in B(\bar z,\rho)$, so that condition \cref{f2} holds at $\bar z$ with $V=B(\bar z,\rho)$ and $\omega=(\sigma\theta\sqrt{n})^{-1}$.
\end{proof}

 We recall now another widely used constraint qualification, namely Slater's, which in most cases is easier to check than MFCQ. Slater's condition holds when there exists $\hat z\in X^*$ such that $g_i(\hat z)< 0$ for $i = 1,\ldots, m$. It is well known that in the convex case Slater's condition implies MFCQ,
 and hence \Cref{t2} holds if we assume Slater's condition
 instead of MFCQ.

\subsection{Linear convergence of cCRM}\label{ss2}

We prove in this subsection the linear convergence of cCRM  under Assumption \ref{eb}. We start deriving the linear convergence and associated rates of MAP and SPM under \ref{eb} with associate rates. These results for MAP and SPM are known to hold under assumptions akin to \ref{eb}, but we include their proofs here for the sake of completeness and self-containment.

 First, remind the definition of Q-linear, Q-superlinear and R-linear convergence.

\begin{definition}[Convergence rate]\label[definition]{d2}
Let $(w^k)_{k\in \na}\subset\re^n$ be a sequence converging to $\bar{w}$. Assume that $w^k\ne \bar{w}$ for all $k\in\na$.
Define \begin{equation}q\coloneqq \limsup_{k\to\infty}\frac{\lV w^{k+1}-\bar{w}\rV}{\lV w^k-\bar{w}\rV}, \quad   \mbox{and}\quad 
r\coloneqq \limsup_{k\to\infty}\lV w^k-\bar{w}\rV^{1/k}.\end{equation} Then, the convergence of $(w^k)_{k\in \na}$ is
\begin{listi}
\item \emph{Q-linear} if $q\in(0,1)$,
\item \emph{Q-superlinear} if $q=0$,
\item \emph{R-linear} if $r\in(0,1)$.
\end{listi}
The values
$q,r$ are called \emph{asymptotic constants} of $(w^k)_{k\in \na}$.
\end{definition}

It is well known that Q-linear convergence implies R-linear convergence (with the same asymptotic constant), but the
converse statement does not hold true~\cite{Ortega:2000}.

We also recall that
for all $z\in \re^n$, we have $T_{\MAP}(z)\coloneqq z_{\MAP}=P_YP_X(z)$, 
$T_{\SPM}(z)\coloneqq z_{\SPM}=\frac{1}{2}(P_X(z)+P_Y(z))$. 

\begin{proposition}[Linear rate of MAP and SPM]\label[proposition]{p1}
Assume that \ref{eb} holds around $\bar z\in X\cap Y$, with $\omega \in (0,1)$ and neighborhood $V$. Let $B$ be a ball centered at $\bar z$ and contained in $V$. Define $\beta\coloneqq \sqrt{1-\omega^2}$. Then,
\begin{equation}\label{e2}
\dist(z_{\MAP},X\cap Y)\le\beta^2 \dist(z,X\cap Y)
\end{equation}
for all $z\in Y\cap B$, and
\begin{equation}\label{e3}
 \dist(z_{\SPM},X\cap Y)\le\left(\frac{1+\beta}{2}\right)\dist(z,X\cap Y)
\end{equation}
for all $z\in B$.
\end{proposition}

\begin{proof}
 We start with the MAP case. Note that
\begin{align}
 \dist^2(z,X\cap Y)& =\lV z-P_{X\cap Y}(z)\rV^2\ge\lV P_X(z)-P_{X\cap Y}(z)\rV^2+\lV z-P_X(z)\rV^2
\\ & \ge 
 \dist^2(P_X(z),X\cap Y)+\omega^2\dist^2(z,X\cap Y),\label{e4}
\end{align} 
 using the firm nonexpansiveness of $P_X$ in the first inequality,
 and Assumption \ref{eb}, together with the fact that $z\in Y\cap B$, so that 
 $\max\{\lV z-P_X(z)\rV,\lV z-P_Y(z)\rV\}=\lV z-P_X(z)\rV$, in the 
 second inequality.
 It follows from \cref{e4} that 
 \begin{equation}\label{e5}
 \dist(P_X(z),X\cap Y)\le\sqrt{1-\omega^2}\dist(z,X\cap Y)=\beta \dist(z,X\cap Y).
 \end{equation}
Observe now that, since $\bar z\in X\cap Y$, we have 
$\lV P_Y(P_X(z))-\bar z\rV\le\lV P_X(z)-\bar z\rV\le\lV z-\bar z\rV$,
so that $P_Y(P_X(z))\in B$. Hence, with the same argument used for
\cref{e5},
 \begin{equation}\label{e6}
 \dist(z_{\MAP}, X\cap Y)=\dist(P_Y(P_X(z)),X\cap Y)\le\beta \dist(P_X(z),X\cap Y),
 \end{equation}
 and \cref{e2} follows combining \cref{e5} and \cref{e6}.
 Now we proceed to establish \cref{e3}. Assume without loss of generality that 
 $\max\{\lV z-P_X(z)\rV, \lV z-P_Y(z)\rV\}=\lV z-P_X(z)\rV$.
 Since $P_X(z)\in B$, we get, with the same argument used for proving \cref{e5},
 \begin{equation}\label{e7}
 \dist(P_X(z),X\cap Y)\le\beta \dist(z,X\cap Y).
 \end{equation}
 
Now, let $\hat s \coloneqq P_{X\cap Y}(z)$. Then,
\begin{align}
\dist(P_Y(z),X\cap Y) & \leq \norm{P_Y(z)-\hat s} = \norm{P_Y(z)-P_Y(\hat s)} \leq \norm{z-\hat s} = \dist(z,X\cap Y),\label{e8}
\end{align}
where the first inequality is due to definition of distance and the last inequality is due to the nonexpansiveness of  $P_Y$. 

Finally, note that
\begin{align}
\dist(z_{\SPM}(z),X\cap Y)& =\dist\left(\frac{1}{2}(P_X(z)+P_Y(z)),X\cap Y\right) \\ & \le \frac{1}{2}\left(\dist(P_X(z),X\cap Y)+\dist(P_Y(z),X\cap Y)\right) \\ 
& 
\le\left(\frac{1+\beta}{2}\right)\dist(z,X\cap Y), \label{ee7}
\end{align}
using the convexity of the distance function to $ X\cap Y$ in the first inequality, and \cref{e7,e8} in the second one, so  the result holds.
\end{proof} 

We remind that MAP generates a sequence $(y^k)_{k\in \na}$ given by 
$y^{k+1}=T_{\MAP}(y^k)$, starting from some $y^0\in Y$, while SPM
generates a sequence $(s^k)_{k\in \na}$ given by $s^{k+1}=T_{\SPM}(s^k)$,
starting from any $s^0\in \re^n$. We have the following corollary of
\Cref{p1}. 

\begin{corollary}[Linear convergence of the distance for MAP and SPM]\label[corollary]{c1} 
Assume that $X\cap Y\ne\emptyset$.
Let $(y^k)_{k\in \na}, (s^k)_{k\in \na}$ be the sequences generated
by MAP and SPM starting from some $y^0\in Y$ and some, $s^0\in \re^n$ respectively. Assume also that $(y^k)_{k\in \na}, (s^k)_{k\in \na}$ are infinite sequences. If \ref{eb} holds at the limits $\bar{y}$ of $(y^k)_{k\in \na}$, $\bar{s}$ of $(s^k)_{k\in \na}$, then the sequences $(\dist(y^k,X\cap Y))_{k\in \na}$,
$(\dist(s^k,X\cap Y))_{k\in \na}\subset\re_+$ converge Q-linearly to $0$, with asymptotic constants given by $\beta^2$, $\frac{1+\beta}{2}$ respectively, where
$\beta=\sqrt{1-\omega^2}$ and $\omega$ is the constant in Assumption \ref{eb}
(with a slight abuse of notation, the sets, and constants guaranteed
by \ref{eb} around both $\bar{y}$ and $\bar{s}$ will be called $V$ and $\omega$).
\end{corollary}

\begin{proof}
Convergence of $(y^k)_{k\in \na}$ and $ (s^k)_{k\in \na}$ to points $\bar{y}\in X\cap Y$ and $\bar{s}\in X\cap Y$ respectively is well known (see, \emph{e.g.}, \cite{Bauschke:1993}). Hence, for large enough $k$, $y^k$ belongs to a ball centered at $\bar{y}$ contained in $V$ and $s^k$ belongs to a ball
centered at $\bar{s}$ contained on $V$.

In view of the definitions of the MAP and SPM sequences, we get from \Cref{p1}, 
\[
\frac{\dist(y^{k+1},X\cap Y)}{\dist(y^k,X\cap Y)}\le\beta^2,\qquad
\frac{\dist(s^{k+1},X\cap Y)}{\dist(s^k,X\cap Y)}\le \frac{1+\beta}{2},
\]
and the result follows from  \Cref{d2}, noting that
\[
\dist(y^k,X\cap Y)\le\beta^{2k}\dist(y^0,X\cap Y),\qquad \dist(s^k,X\cap Y)\le\left[\frac{1+\beta}{2}\right]^k \dist(s^0, X\cap Y),
\]
so that both sequences converge to $0$, since $\beta\in (0,1)$, because $\omega\in (0,1)$.
\end{proof}

We state now a result on linear convergence of Fej\'er monotone sequences.

\begin{proposition}[Fejér monotonicity and linear convergence]\label[proposition]{p2}
If a sequence $(w^k)_{k\in \na}\subset\re^n$ is Fej\'er monotone with respect to a closed convex set  $M\subset\re^n$ and the scalar sequence $(\dist(w^k,M))_{k\in\na}$ converges Q-linearly
to $0$, then $(w^k)_{k\in \na}$ converges R-linearly to a point $\bar{w}\in M$.
\end{proposition}

\begin{proof} 
See Lemma 3.4 in \cite{Arefidamghani:2021}, cf. Theorem 5.12 in \cite{Bauschke:2017a}.
\end{proof}

\begin{corollary}[Linear convergence of MAP and SPM]\label[corollary]{c2}
Assume that $X\cap Y\ne\emptyset$.
Let $(y^k)_{k\in \na}, (s^k)_{k\in \na}$  be the sequences generated
by MAP and SPM starting from some $y^0\in Y$ and some $s^0\in\re^n$ respectively. Assume also that $(y^k)_{k\in \na}, (s^k)_{k\in \na}$ are infinite sequences.
If \ref{eb} holds around the limits $\bar{y}$ of $(y^k)_{k\in \na}$ and $\bar{s}$ of $(s^k)_{k\in \na}$ with $\omega \in (0,1)$ and neighborhood $V$, then the sequences $(y^k)_{k\in \na}$,
$(s^k)_{k\in \na}$ converge R-linearly, with asymptotic constants bounded above by $\beta^2$, $\frac{1+\beta}{2}$ respectively. Here,
$\beta\coloneqq\sqrt{1-\omega^2}$, and $\omega$ is the constant in Assumption \ref{eb}.
\end{corollary}
\begin{proof}
The fact that $(y^k)_{k\in \na}, (s^k)_{k\in \na}$ are Fej\'er monotone with respect to $X\cap Y$ is well known and is an immediate consequence of the firm nonexpansiveness of $P_X, P_Y$. Then, the result follows from \Cref{c1,p2}.
\end{proof}

We remark that the MAP sequence converges quite faster than the SPM  one, in terms of the upper bound of their asymptotic error constants, since $\beta^2\le\frac{1+\beta}{2}$, and the difference becomes more significant as $\beta$ approaches $0$. However, SPM may be faster in the presence of parallel processors. Observe first that the expensive steps in both algorithms are the computation of $P_X,P_Y$. In the absence of parallel processors, each iteration of either MAP or SPM requires two projections (one onto $X$ and one onto $Y$), and the work per iteration is about the same for both methods. On the other hand, in MAP both projections must be computed sequentially, while in SPM they can be computed simultaneously if two parallel processors are available. In such a situation, one step of MAP is equivalent to two of SPM, and the asymptotic constants, in terms of the number of projections, become $\beta$ and $\frac{1+\beta}{2}$ respectively. Yet, MAP wins over 
SPM, which turns out to be indeed competitive when used for finding a point in the intersection of $m$ sets ($m\ge 3$). In this case, assuming that $m$ parallel processors are available, one step of MAP is equivalent to $m$ steps of SPM. For more information on sequential versus simultaneous methods; see \cite{Censor:1998}.

 
We establish now the R-linear convergence of the cCRM sequence under \ref{eb}, and give an upper bound for the asymptotic constant
 in terms of the constant $\omega$ in \ref{eb}.
 
Recall that, given $z\in\re^n$,
we denote 
$z_{\Cent}=\frac{1}{2}(z_{\MAP}+P_X(z_{\MAP}))$,
$T(z)= {\rm circ}\{z_{\Cent},R_X(z_{\Cent}), R_Y(z_{\Cent})\}$, with $z_{\MAP}$ as in the previously defined. The cCRM sequence $(z^k)_{k\in \na}\subset\re^n$  is given by $z^{k+1}=T(z^k)$, for any $z^0\in\re^n$.
 
Next, we present  an upper bound for the advance ratio of a cCRM step towards the solution set.

\begin{proposition}[Linear convergence of the distance for cCRM]\label[proposition]{p3}
  Assume that \ref{eb} holds around $\bar z\in X\cap Y$ with $\omega \in (0,1)$ and neighborhood $V$. Let $B$ be a ball centered at $\bar z$ and contained in $V$. Define $\beta=\sqrt{1-\omega^2}$. Then,
  \begin{equation}\label{e9}
  \dist(T(z),X\cap Y)\le\beta^2\left(\frac{1+\beta}{2}\right)\dist(z,X\cap Y),
  \end{equation}
  for all $z\in B$.
  \end{proposition}
  \begin{proof}
  By nonexpansiveness of $P_X$, we get similarly as in \cref{e8} that
  \begin{equation}\label{eb0}
  \dist(P_X(z),X\cap Y)\le \dist(z,X\cap Y).
  \end{equation}
  Note that nonexpansiveness of $P_X, P_Y$ imply that
  $\lV z_{\MAP}-\bar z\lV\le\rV z-\bar z\rV$ for all $\bar z \in X\cap Y$, so that
  $z_{\MAP}\in B$ whenever $z\in B$. With the same argument used in the proof of \Cref{p1} for
  establishing \cref{e5}, we get
  \begin{equation}\label{eb1}
  \dist(z_{\MAP},X\cap Y)\le\beta\dist(P_X(z),X\cap Y),
  \end{equation}
  and
  \begin{equation}\label{eb2}
  \dist(P_X(z_{\MAP}),X\cap Y)\le\beta^2 \dist(P_X(z),X\cap Y).
  \end{equation}
  Again, nonexpansivenes of $P_X$ ensures that $P_X(z_{\MAP})$ is closer than $z_{\MAP}$ to any point in $X\cap Y$, so that $P_X(z_{\MAP})$ belongs to $B$ whenever $z\in B$. Hence, with the same argument as in the proof of \Cref{p1} for establishing \cref{ee7}, we get
  \begin{align}
  \dist(z_{\Cent},X\cap Y)& =
  \dist\left(\frac{1}{2}(z_{\MAP}+P_X(z_{\MAP})),X\cap Y\right) \\ & \le\frac{1}{2}
  \left(\dist(z_{\MAP},X\cap Y)+\dist(P_X(z_{\MAP}),X\cap Y\right) )
  \\
  &\le\frac{1}{2}\beta(1+\beta)
  \dist(P_X(z),X\cap Y),\label{eb3}
  \end{align}
  using \cref{eb1,eb2} in the last inequality.
  We invoke now  \cref{lemma:ZcCharacterization} for $z_{\Cent}$, which implies
  that $T(z)=P_{S^{z_{\Cent}}_X\cap S^{z_{\Cent}}_Y}(z_{\Cent})$. Recall that $S^{z_{\Cent}}_X=\{w\in \re^n \mid \scal{w-P_X(z_{\Cent})}{z_{\Cent} - P_X(z_{\Cent})}\leq 0\}$ and $S^{z_{\Cent}}_Y=\{w\in \re^n \mid \scal{w-P_Y(z_{\Cent})}{z_{\Cent} - P_Y(z_{\Cent})}\leq 0\}$, so we easily get that $P_X(z_{\MAP})=P_{S^{z_{\Cent}}_X}(z_{\MAP})$. Since $T(z)\in S^{z_{\Cent}}_X\cap S^{z_{\Cent}}_Y\subset S^{z_{\Cent}}_X$, we get that
  \begin{equation}\label{eb4} 
  \lV T(z)-z_{\Cent}\rV\ge\lV P_X(z_{\Cent})-z_{\Cent}\rV.
  \end{equation}
 
  We claim that $\lV P_X(z_{\Cent})-z_{\Cent}\rV\ge\lV P_Y(z_{\Cent})-z_{\Cent}\rV$.
  Indeed, since $z_{\Cent}$ is the midpoint between $P_X(z_{\MAP})$ and
  $z_{\MAP}$, we get that $P_X(z_{\Cent})=P_X(z_{\MAP})$ and also
  \[
  \lV P_X(z_{\Cent})-z_{\Cent}\rV=\lV P_X(z_{\MAP})-z_{\Cent}\rV=\lV z_{\MAP}-z_{\Cent}\rV\ge 
  \lV P_Y(z_{\Cent})-z_{\Cent}\rV,
  \]
  using the fact that $z_{\MAP}\in Y$ in the last inequality. The claim holds, and hence 
  \[
  \max\{\lV P_X(z_{\Cent})-z_{\Cent}\rV,\lV P_Y(z_{\Cent})-z_{\Cent}\rV\}=
  \lV P_X(z_{\Cent})-z_{\Cent}\rV.
  \]
  
  It follows from \Cref{lemma-quasinonexp-pCRM-combined} that
  $z_{\Cent}$ belongs to $B$ whenever $z\in B$. Hence, we invoke \ref{eb}, which implies that
  \begin{equation}\label{eb5}
  \lV P_X(z_{\Cent})-z_{\Cent}\rV\ge\omega \dist(z_{\Cent}, X\cap Y).
  \end{equation}
  Combining \cref{eb4,eb5}
  we get
  \begin{equation}\label{eb6}
  \lV T(z)-z_{\Cent}\rV\ge\omega \dist(z_{\Cent}, X\cap Y).
  \end{equation}
  Moreover, since
  $T(z)=P_{S^{z_{\Cent}}_X\cap S^{z_{\Cent}}_Y}(z_{\Cent})$ and $X\cap Y\subset S^{z_{\Cent}}_X\cap S^{z_{\Cent}}_Y$,
  we have that 
  \[
  \lV T(z)-s\rV^2\le\lV z_{\Cent}-s\rV^2-\lV T(z)-z_{\Cent}\rV^2
  \]
  for all $s\in X\cap Y$, so that
  \begin{align}
  \dist^2(T(z),X\cap Y)& \le \dist^2(z_{\Cent},X\cap Y)-\lV T(z)-z_{\Cent}\rV^2 \\ & \le
  \dist^2(z_{\Cent},X\cap Y)-\omega^2\dist^2(z_{\Cent}, X\cap Y)=\beta^2\dist^2(z_{\Cent}, X\cap Y),\label{eb7}
  \end{align}
  using \cref{eb6} in the last inequality. It follows from \cref{eb7}
  that 
  \begin{equation}\label{eb8}
  \dist(T(z),X\cap Y)\le\beta \dist(z_{\Cent},X\cap Y).
  \end{equation}

  Combining \cref{eb0,eb3,eb8},
  we get
  \begin{align}
  \dist(T(z),X\cap Y)& \le\beta \dist(z_{\Cent},X\cap Y)\le\frac{1}{2}\beta^2(1+\beta)
  \dist(P_X(z),X\cap Y)
  \\ & \le\frac{1}{2}\beta^2(1+\beta)\dist(z,X\cap Y),
  \end{align}
  which establishes the result.
  \end{proof}

In the following, we get from \Cref{p3} the R-linear convergence result for the cCRM sequence, in a way similar to the proof of R-linear convergence
of the MAP and SPM in \Cref{c2}.

\begin{theorem}[Linear convergence of cCRM]\label{t1}
    Assume that $X\cap Y\ne\emptyset$.
    Let $(z^k)_{k\in\na}$ be the sequence generated
    by cCRM starting from some $z^0\in\re^n$. Assume also that $(z^k)_{k\in\na}$ is an infinite sequence. If \ref{eb} holds around the limit $\bar z$ of $(z^k)_{k\in\na}$ with $\omega \in (0,1)$ and neighborhood $V$, then the sequence $(z^k)_{k\in\na}$
    converges R-linearly to some point in $X\cap Y$, with asymptotic constant bounded above by $\beta^2\left(\frac{1+\beta}{2}\right)$, where $\beta=\sqrt{1-\omega^2}$ and $\omega$ is the constant in Assumption \ref{eb}.
    \end{theorem}
\begin{proof}
Convergence of $(z^k)_{k\in\na}$ to a point $\bar z\in X\cap Y$  follows from \cref{thm:pCRM-centralized}. Hence, for large enough $k$, $z^k$ belongs to the ball centered at $\bar z$ and contained in $V$, whose existence in ensured by Assumption \ref{eb}. 

We recall that the cCRM sequence is defined as $z^{k+1}=T(z^k)$,
so that it follows from \Cref{p3} that
\begin{equation}\label{eb9}
\frac{\dist(z^{k+1}, X\cap Y)}{\dist(z^k,X\cap Y)}\le\beta^2\left(\frac{1+\beta}{2}\right).
\end{equation}
Since $\omega\in(0,1)$ implies that $\beta^2(\frac{1+\beta}{2})\in (0,1)$,
it follows immediately from \cref{eb9} that the scalar sequence 
$(\dist(z^k,X\cap Y))_{k\in\na}$ converges Q-linearly to $0$ with asymptotic constant bounded above by $\beta^2\left(\frac{1+\beta}{2}\right)$.

Finally, recall that sequence $(z^k)_{k\in\na}$ is Fej\'er monotone with respect to $X\cap Y$, due to \cref{thm:pCRM-centralized}. The R-linear convergence of $(z^k)_{k\in\na}$
to some point in $X\cap Y$ and the value of the upper bound of the asymptotic constant follow then from \Cref{p2}.
\end{proof}

In addition to establishing linear convergence of cCRM under the error bound condition, \cref{t1} provides an upper bound for cCRM's linear rate that is the product of the reduction factors of MAP, namely $\beta^2$,  and SPM, namely $\frac{1+\beta}{2}$. 
In view of the superlinear convergence result that we are going to present in the next subsection and our numerical results, the actual asymptotic constant for cCRM seems to be quite smaller than the upper bound presented in \Cref{p2,t1}. The issue of improving this upper bound deserves further research.

\subsection{Superlinear convergence of cCRM}\label{sec:superlinearcCRM}

In previous works it was discussed that circumcentering-type schemes have a Newtonian flavor, meaning that superlinear convergence could be expected. This was derived in a very limited setting, namely, for CRM as a root finder of smooth convex functions in \cite[Corollary 4.11]{Arefidamghani:2021}. Another reference that addresses particular examples showing superlinear convergence of circumcentering techniques is \cite{Dizon:2022}. With our novel approach of centralizing CRM, we are able to cover a way broader class of convex feasibility problems, including nonsmooth ones,  for which we get superlinear convergence. Our result states that if the boundaries of $X$ and $Y$ are locally differentiable manifolds, and the interior of $X\cap Y$ is nonemtpy, cCRM converges superlinearly to a solution of the CFP \cref{eq:CFP}. The main theorem relies on forthcoming lemma and  proposition.

\begin{lemma}[Superlinear convergence of the distance for cCRM]\label[lemma]{lemma:superlinearConvDist}
  Let $X,Y\subset \re^n$ be closed, convex and  suppose $X\cap Y \neq \emptyset$. Let $(z^k)_{k\in\na}$ be the sequence generated  by cCRM starting from some $z^0\in\re^n$, and converging to a point $\bar z \in X\cap Y$. Assume that the interior of $X\cap Y$ is nonempty, and that the boundaries of $X$ and $Y$ are differentiable manifolds in a neighborhood of  $\bar z$.  Then, the scalar sequence  $(\dist(z^k, X\cap Y))_{k\in \na}$ converges to zero superlinearly. 
\end{lemma}

\begin{proof}
  If the sequence  $(z^k)_{k\in\na}$ is finite then the announced result follows trivially.  
    So, let us assume that it is infinite. 

    In order to prove the superlinear convergence stated in the lemma, \emph{i.e.}, 
    \[\label{eq:superlindis1}
      \lim_{k\to \infty} \frac{\dist(z^{k+1}, X\cap Y)}{\dist(z^k, X\cap Y)}   = 0,
    \]
it suffices to show that 
\[\label{eq:superlindis2}
  \lim_{k\to \infty} \frac{\dist(z^{k+1}, X\cap Y)}{\dist(z_{\Cent}^k, X\cap Y)}   = 0,
\]
where $z_{\Cent}^{k}$ is the centralized point associated to $z^{k}$. Indeed, if we prove \cref{eq:superlindis2}, then \cref{eq:superlindis1} follows since  $\dist(z_{\Cent}^k,X\cap Y) \leq \dist(z^k,X\cap Y)$, by the Fejér monotonicity of the centralization procedure guaranteed in \Cref{lemma-quasinonexp-Cent}. 

Recall that
\[
  z^{k+1} =   \pCRMOp(z_{\Cent}^k) = \circum\{ z_{\Cent}^k, R_X(z_{\Cent}^k), R_Y(z_{\Cent}^k)\},
\]
and because $z_{\Cent}^k$ is strictly centralized,  $z^{k+1} $ was characterized in \Cref{lemma:ZcCharacterization} as $P_{H_X^{k}\cap H_Y^{k}}(z_{\Cent}^k)$, where  $H_X^k \coloneqq H_X^{z_{\Cent}^k}$ and $H_Y^k \coloneqq H_Y^{z_{\Cent}^k}$ are the hyperplanes passing through $P_{X}(z_{\Cent}^k)$ and $P_{Y}(z_{\Cent}^k)$ that are orthogonal to $z_{\Cent}^k - P_{X}(z_{\Cent}^k)$ and $z_{\Cent}^k - P_{X}(z_{\Cent}^k)$, respectively. Therefore, since $z^{k+1} \in H_X^{k}\cap H_Y^{k} $, Pythagoras gives us
\begin{align}
    \norm{z^{k+1} - P_{X}(z_{\Cent}^{k})}^2 = \norm{z^{k+1} - z_{\Cent}^{k}}^2 - \norm{z^{k}_C - P_{X}(z_{\Cent}^{k})}^2 ,   
\end{align}
which implies
\[\label{eq:foo3_super}
  \norm{z^{k+1} - P_{X}(z_{\Cent}^{k})} \leq \norm{z^{k+1} - z_{\Cent}^{k}}.\]  
By the same token, we also get 
  \[\label{eq:foo4_super}
    \norm{z^{k+1} - P_{Y}(z_{\Cent}^{k})} \leq \norm{z^{k+1} - z_{\Cent}^{k}}.
\]

The hypothesis    $\inte (X\cap Y)\neq \emptyset$  gives us an error bound    $\omega \in (0,1)$  as in \ref{eb}. In particular, there exists  $\hat k\in \na$ such that 
\[\label{eq:errorboundzk}
 \omega \dist(z
 ^{k},X\cap Y) \leq   \max \{\dist(z
 ^{k}, X), \dist(z
 ^{k}, Y)\},
\]
for all $k\geq \hat k$; see Corollary 5.14 of \cite{Bauschke:1996}.

In addition, the nonemptiness of the interior of $ X\cap Y$, together with the hypothesis that the boundaries of $X$  and $Y$ are locally differentiable manifolds, imply that these manifolds have dimension  $n-1$.

Now, we claim that if $M\subset \RR^n$ is a differentiable manifold of dimension  $n-1$, $\bar z$ belongs to $M\subset \RR^{n}$
and $z\in \RR^n$ belongs to the tangent hyperplane to $M$ at $\bar z$, say $T_M(\bar z)$, then
\begin{equation}\label{v1}
\lim_{z\to\bar z}\frac{\dist(z,M)}{\lV z-\bar z\rV}=0.
\end{equation}
This result follows, with an elementary analysis argument, from the well known fact that
$M$ can be locally written as $\{z\in \re^n \mid g(z)=0\}$ for some function $g:\re^n\to\re$ of class ${\cal C}^1$
with $g(\bar z)=0, \nabla g(\bar z)\ne 0$, so that 
$T_M(\bar z)\coloneqq \{z\in\re^n\mid  \scal{\nabla g(\bar z)}{z-\bar z}=0\}$.

Thus, the hyperplanes $H_X^{k}$ and $H_Y^{k}$, which have dimension $n-1$ each, are tangent to the manifolds~\cite[Theorems 23.2 and 25.1]{Rockafellar:1997}, for all large $k$,  respectively at $P_{X}(z_{\Cent}^{k})$ and $P_{Y}(z_{\Cent}^{k})$. 
Since $z^{k+1}$ lies in both $H_X^{k}$ and $H_Y^{k}$, we have, in view of \cref{v1},  
\[\label{eq:foo1_super}
  \lim_{k\to \infty} \frac{\dist(z^{k+1},X)}{\norm{z^{k+1} - P_{X}(z_{\Cent}^{k})}} = 0
\]
and
\[\label{eq:foo2_super}
  \lim_{k\to \infty} \frac{\dist(z^{k+1},Y)}{\norm{z^{k+1} - P_{Y}(z_{\Cent}^{k})}} = 0.
\]

From \cref{lemma-quasinonexpCRM}, it holds that 
\(\norm{z^{k+1} - z_{\Cent}^{k}} \leq \dist(z_{\Cent}^{k}, X\cap Y),
\) which combined with the previous inequalities \cref{eq:foo3_super,eq:foo4_super}, implies that 
\[
\norm{z^{k+1} - P_{X}(z_{\Cent}^{k})} \leq   \dist(z_{\Cent}^{k}, X\cap Y) 
\]
and
\[
\norm{z^{k+1} - P_{Y}(z_{\Cent}^{k})} \leq   \dist(z_{\Cent}^{k}, X\cap Y). 
\]
Hence, from \cref{eq:foo1_super,eq:foo2_super} we get  
\[
  \lim_{k\to \infty} \frac{\dist(z^{k+1},X)}{\dist(z_{\Cent}^{k}, X\cap Y)} = 0
\]
and
\[
  \lim_{k\to \infty} \frac{\dist(z^{k+1},Y)}{\dist(z_{\Cent}^{k}, X\cap Y)} = 0.
\]

Note that \cref{eq:errorboundzk} yields 
\[
 \omega \frac{\dist(z
 ^{k},X\cap Y)}{\dist(z_{\Cent}^{k}, X\cap Y)} \leq   \max \left\{\frac{\dist(z
 ^{k}, X)}{\dist(z_{\Cent}^{k}, X\cap Y)}, \frac{\dist(z
 ^{k}, Y)}{\dist(z_{\Cent}^{k}, X\cap Y)} \right\}.
\]
Taking limits as $k\to \infty$,  we obtain \cref{eq:superlindis2} and the proof is completed.  
\end{proof}

Similarly to \Cref{p2}, we establish now a result stating that if a sequence  is Fej\'er monotone with respect to a given closed convex set and the distance of the sequence to that said set converges superlinearly to zero, then the sequence itself  also converges superlinearly to a point in the corresponding set.

\begin{proposition}[Fejér monotonicity and superlinear convergence]\label[proposition]{prop:superlinearFejer}
  If a sequence $(w^k)_{k\in \na}\subset\re^n$ is Fej\'er monotone with respect to a closed convex set  $M\subset\re^n$ and the scalar sequence $(\dist(w^k,M))_{k\in\na}$ converges superlinearly
  to $0$, then $(w^k)_{k\in \na}$ converges superlinearly to a point $\bar{w}\in M$.
  \end{proposition}
  
  \begin{proof} 
Let $(w^k)_{k\in \na}\subset\re^n$ be a  Fej\'er monotone sequence with respect to the closed convex set $M\subset\re^n$. Suppose that the scalar sequence 
$(\dist(w^k,M))_{k\in\na}$ converges superlinearly
to $0$, that is, that 
\[\label{eq:superlinear_dist_Fejer}
  \lim_{k\to \infty}
\frac{\dist(w^{k+1}, M)}{\dist(w^k, M)} = 0.
\]
Assumption \cref{eq:superlinear_dist_Fejer} promptly yields $\dist(w^{k}, M) \to 0$, and then appealing to  \cite[Theorem 5.11]{Bauschke:2017a} we get that $(w^k)_{k\in \na}$ has a limit point $\bar w \in M$.

Note now that, for any $m,k\in \na$, we have 
    \begin{align}
      \norm{w^{k+1} - w^{m+k+1}} & \leq  \norm{w^{k+1} - P_{M}(w^{k+1})} + \norm{ w^{m+k+1} - P_{M}(w^{k+1})}     \\ 
      & \leq \norm{w^{k+1} - P_{M}(w^{k+1})}  + \norm{ w^{k+1} - P_{M}(w^{k+1}) }  \\
      &  = 2\dist(w^{k+1}, M),\label{eq:prop_superlinearfoo1} \end{align}
        where, in the second inequality, we use $m$ times the Fejér monotonicity  of sequence   $(w^k)_{k\in\na}$.
        Taking the limit as $m\to \infty$ in \cref{eq:prop_superlinearfoo1} gives us, for all $k\in \na$,  
        \begin{align}
          \norm{w^{k+1} - \bar w} & \leq  2\dist(w^{k+1}, M).
          \end{align}
      Since $\dist(w^k, M) \leq \norm{w^k - \bar w}$, we conclude that 
      \begin{align}
        \frac{\norm{w^{k+1} - \bar w}}{ \norm{w^k - \bar w}} & \leq  2\frac{\dist(w^{k+1}, M)}{\dist(w^k, M)}.
      \end{align}

Finally, the limit,  as $k \to \infty$, of the right-hand side of this inequality goes to zero, due to   \cref{eq:superlinear_dist_Fejer}. Therefore, we get the superlinear convergence of $(w^k)_{k\in \na}$ to $\bar w$, as required.
  \end{proof}

We can now prove the superlinear convergence cCRM, when the interior of the intersection of $X$ and $Y$ is nonempty, and the boundaries  of $X$ and $Y$ are locally smooth manifolds.

\begin{theorem}[Superlinear convergence of cCRM]\label{theorem:superlinearConv}
  Let $X,Y\subset \re^n$ be closed, convex and  suppose $X\cap Y \neq \emptyset$. Let $(z^k)_{k\in\na}$ be the sequence generated  by cCRM starting from some $z^0\in\re^n$, and converging to a point $\bar z \in X\cap Y$. Assume that the interior of $X\cap Y$ is nonempty, and that the boundaries of $X$ and $Y$ are differentiable manifolds in a neighborhood of  $\bar z$.  Then, $(z^k)_{k\in\na}$ converges to $\bar z$ superlinearly. 
\end{theorem}

\begin{proof}
  The theorem is a direct  consequence of the Fejér monotonicity of  $(z^k)_{k\in\na}$ with respect to $X\cap Y$ given in \Cref{thm:pCRM-centralized}, together with  \Cref{lemma:superlinearConvDist,prop:superlinearFejer}.
\end{proof}

To the best of our knowledge, \Cref{theorem:superlinearConv} is so far the strongest result regarding the convergence rate of circumcenter-type methods.   

\section{Numerical experiments}\label{s6}



In this section,  we study the performance of cCRM by means of numerical comparisons with two other methods, namely, MAP and CRMprod. Since we are not using parallel computation, it is well-known that SPM underperforms when compared with MAP, thus we do not include SPM in our report. 

The experiments address two classes of non-affine convex intersection problems. We first seek a common point of two ellipsoids, and then we consider the problem of finding a point in the intersection of a second order cone and a polyhedron.   In the first class of experiments we are able to illustrate the superlinear convergence of cCRM stated in \Cref{theorem:superlinearConv} since the boundary of an ellipsoid is a differentiable manifold. The numerical results for the second class of problems show much faster convergence of cCRM in comparison to MAP and CRMprod.

We remark that each iteration of cCRM requires four orthogonal projections. In principle, they seem to be five: we compute three sequential projections from $z=z^k$ onto
$X, Y$ and again onto $X$ for obtaining $z_{\MAP}$, and then we project $z_{\MAP}$ onto $X$ and $Y$ for getting $z_{\Cent}$. However, as observed above, $P_X(z_{\Cent})=P_X(z_{\MAP})$, because $z_{\Cent}$ is in the segment between
$z_{\MAP}$ and $P_X(z_{\MAP})$. On the other hand, both MAP and CRMprod require just two projections per iteration. Therefore, for a fair comparison we count the number of projections required by each method in order to achieve the desired precision. We also mention that the cost of the step from $z_{\Cent}$ to $T(z)$ is indeed negligible because the computation of a circumcenter reduces to solving a  $2 \times 2$ system of linear equations. Explicit formulas for computing even more general  circumcenters are presented in \cite[Theorem 4.1]{Bauschke:2018} and \cite[Section 3]{Behling:2018a}.

The computational experiments were carried out on an Intel Xeon W-2133 3.60GHz with 32GB of RAM running Ubuntu 20.04. The codes were implemented   in \texttt{Julia}  programming language v1.8~\cite{Bezanson:2017}, and  are  available at \url{https://github.com/lrsantos11/CRM-CFP}.

\subsection{Intersection of Two Ellipsoids}\label{sec:intertwoellipsoids}

In this subsection, we consider  cCRM, MAP and CRMprod for solving the particular CFP  of finding a common point  in the intersection of two ellipsoids, that is, finding 
\[\label{eq:ellipsoids} \bar z \in \Eset \coloneqq \Eset_1\cap \Eset_2 \subset \re^n,\] where each ellipsoid  $\Eset_i$ is set as
\[\Eset_{i}\coloneqq \left\{z\in \re^n \mid g_i(z) \leq 0 \right\}, \text{ for } i=1,2, \]
where $g_i:\re^n\to\re$ is defined as $g_i(z) =  \scal{z}{ A_{i} z} +2 \scal{z}{b^{i}} -  \alpha_{i}$, each $A_{i}$ is a symmetric positive definite matrix, $b^{i}$ is an $n$-vector, and  $\alpha_{i}$ is a positive scalar.  Problem \cref{eq:ellipsoids}  has importance on its own; see \cite{Lin:2004,Jia:2017}.

To run the tests, we randomly produce instances of \cref{eq:ellipsoids} with the following procedure. We first form ellipsoid $\Eset_1$ by generating  a matrix $A_1$ of the form $A_1 = \gamma \Id +B_1^\top B_1$, with  $B_1 \in\re^{n\times n}$, $\gamma \in \re_{++}$.   Matrix $B_1$ is  sparse  with sparsity density  $p=2 n^{-1}$ and its components are sampled from the standard normal distribution.  Vector $b^1$ is sampled from the uniform distribution in the interval $[0,1]$. We then choose each $\alpha_1$ so that $\alpha_1 > (b^1)^\top A_1b^1$, which ensures that $0$ belongs to $\Eset_1$. Next, we construct $\Eset_2$ by  randomly choosing its center  $c_2$ outside $\Eset_1$, then we project $c_2$ onto  $\Eset_1$. We define    $d\coloneqq \lambda(P_{\Eset_1}(c_2) - c_2)$ as the principal axis of the ellipsoid correspondent to the least value of semi-axis. In order to get this, we form a diagonal matrix $\Lambda \coloneqq \diag(\norm{d}, u)$, where $u\in \re^{n-1}$ is a vector whose components are positive and have values greater than $\norm{d}$, and an orthogonal matrix $Q$ where the first column is $d/\norm{d}$. Finally, we set $A_2 \coloneqq Q^\top \Lambda^2 Q$ and $g_2(z)\coloneqq \scal{z-c_2}{A_2(z-c_2)}$. 
For all methods and instances, the initial point $z^0$ is sampled from the standard normal distribution guaranteeing that its norm is at least  \num{5} and also that $z^0 \notin \Eset_1\cap\Eset_2$.

The projections onto ellipsoids are computed using an alternating direction method of multipliers (ADMM) built suited for this end~\cite{Jia:2017}. 
For testing the methods, we generate two types of instances, with $n=\num{100}$. In the first group, $\Eset_1\cap\Eset_2$ has nonempty interior, and in the second one  $\Eset_1\cap\Eset_2$ is a singleton.

\subsubsection{$\Eset_1\cap\Eset_2$ with nonempty interior}\label{ss6.1}

We first randomly generate \num{30} instances where the intersection has nonempty interior. For that, we set $\lambda = 1.1$ such that $d\coloneqq 1.1\left(P_{\Eset_1}(c_2) - c_2\right)$. The rationale here is that though $\Eset_1\cap\Eset_2$ has indeed nonempty interior, the \num{1.1} multiplying factor guarantees that the intersection is \emph{not too large} (otherwise the problems are very easy to solve).

In the examples of this subsection the fact that the ellipsoids have their intersection with nonempty interior guarantees that the Slater condition holds. Moreover, the boundaries of ellipsoids are differentiable manifolds, so \Cref{theorem:superlinearConv} applies, and we get superlinear convergence of cCRM, which is not achievable by MAP. Therefore, there is no surprise in seeing in \Cref{fig:pp-Ellipsoid} and
 \Cref{table:ellipsoid}  cCRM vastly outperforming MAP. We note that cCRM also outperforms CRMprod by far.

Remind that we are considering the total number of projections employed by each method to achieve convergence upon the desired tolerance using the following criteria. For each sequence $({w}^k)_{k\in\na}$ yielded by the considered methods, we use as tolerance $\varepsilon \coloneqq \num{e-6}$ and as stopping criteria the gap distance 
 \[
 \lVert P_{\Eset_1}({w}^{k}) - {w}^k\rVert <  \varepsilon .
 \]
 This is a reliable measure of infeasibility because the Slater condition implies that Assumption \ref{eb} holds.
 We also set a budget of \num{10000} total number of allowed projections for each method.

\Cref{fig:pp-Ellipsoid} is a performance profile~\cite{Dolan:2002}.  Performance profiles allow one to benchmark different methods on a  set of problems with respect to a performance measure (in our case, the number of projections). The vertical axis indicates the percentage of problems solved, while the horizontal axis indicates, in log-scale, the corresponding factor of the performance index used by the best solver. The picture clearly shows that cCRM always does better than the other two methods, being faster and more robust (MAP and CRMprod did not solve one instance). In addition,  MAP and CRMprod took more than $2^6$ times the number of projections that cCRM used to solve all problems.

\begin{figure}[hbt!]
  \centering
  \includegraphics[scale=0.7]{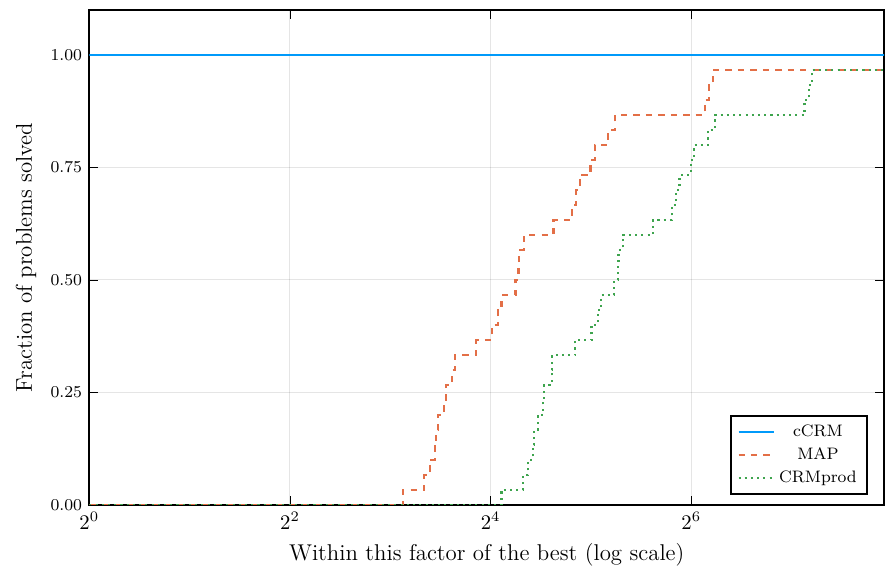}
  \caption{Performance profile of experiments with ellipsoidal feasibility.}
  \label{fig:pp-Ellipsoid}
\end{figure}

We conclude this examination by discussing \Cref{table:ellipsoid}, which presents the following descriptive statistics of the benchmark: mean $\pm$  standard deviation (std), median, minimum (min), and maximum (max) of total projections count. As expected,  cCRM outstandingly handles both MAP and CRMprod. In fact, the superlinear convergence of cCRM is translated numerically as it takes, on average,  \num{31} times fewer projections than MAP while taking almost \num{49} times fewer projections than CRMprod. 

\begin{table}[ht!]
  \centering
  \caption{Statistics of the experiments  (in number of projections) with  $\inte \Eset \neq \emptyset$.}
  \label{table:ellipsoid}
  \sisetup{
table-parse-only,
table-figures-decimal = 2,
table-format = +3.1e-1,
table-auto-round,
separate-uncertainty
%
}
\begin{tabular}{lSSSS}
\toprule  
& {\textbf{mean $\pm$ std}} &   {\textbf{median}}   &  {\textbf{min}} &   {\textbf{max}}     \\
\cmidrule(lr){2-5}
cCRM &    \num{26.13+-44.28}     &    16.0 &  16 &   260\\
MAP &     \num{817.20+-1802.54}    &   308.0 & 140 & 10000\\
CRMprod & \num{1295.07+-1912.92}   &   610.0 & 276 & 10000\\
\bottomrule
\end{tabular}
  
\end{table}

\subsubsection{$\Eset_1\cap\Eset_2$ is a singleton}\label{ss6.2}

We now run the methods in  \num{15} more challenging instances. 
For that, we set $\lambda = 1.0$, that is, $d\coloneqq P_{\Eset_1}(c_2) - c_2$ is the principal axis related to the smallest  semi-axis of $\Eset_2$. Therefore,  $\Eset_1\cap\Eset_2 = \{ P_{\Eset_1}(c_2)\}$, \emph{i.e.}, we have a singleton. This is now a  harder problem to solve because of the lack of regularity since the interior of the intersection is empty (there is no Slater point in it and hence \ref{eb} is not fulfilled). Thus, we use as tolerance $\varepsilon \coloneqq \num{e-3}$. 
Note that we have in hand  the unique solution $\bar z \coloneqq P_{\Eset_1}(c_2)$, so we set as stopping criteria the distance to the solution, that is, we stop whenever
\[
\lVert \bar{z} - {z}^k\rVert <  \varepsilon .
\]
We also allow the methods to go further and take up to  \num{500000} total number of projections.

This time, MAP and CRMprod could not achieve the desired tolerance (\num{e-3}) in any of the instances, that is, in all cases those methods reached the maximum number of projections allowed (\num{500000}). On average, the distance to the solution with this amount of projections for MAP was \num{1.2765e-2}$\pm$\num{2.45e-3}. Those findings were, in fact, expected, for MAP was shown to have, at best, sublinear convergence rate when there is no Lipschitzian regularity between the underlying sets \cite[Theorem 2.2]{Drusvyatskiy:2016}. 

In contrast to MAP and CRMprod,  cCRM seems to converge linearly, even with  \ref{eb} failing to hold. We report that our proposed method took, on average, {\num{28770.1+-40343.4}} projections to stop to the tolerance, whereas the total projections for each instance ranges  from {\num{1068}} (minimum) to {\num{143368}} (maximum). Here, the apparent linear convergence of cCRM over the sublinear behavior of MAP  agrees with the examples in \cite{Arefidamghani:2021}, where in an affine/convex context without error bound, CRM  converges linearly, opposed to MAP's sublinear convergence. This situation certainly deserves a deeper study.

\subsection{Intersection of a second order cone with a polyhedron}\label{sec:inter_cones}

In the following experiments, we want to find $\bar z \in \re^{n}$ that lies in   
\[
\label{eq:ConicSystem}
\Sset \coloneqq \Omega \cap \SOC_n,
\]
where  $\Omega \coloneqq \{w\in \RR^n \mid Aw \leq b\}$, with  $A\in\RR^{m\times n}$ and $b\in\RR^{m}$, is a \emph{polyhedron}  and $\SOC_{n}$ is the  \emph{standard second-order cone of dimension $n$} defined as
\[
\SOC_{n} \coloneqq  \{(t,u)\in \RR^{n}\mid u\in\RR^{n-1}, t\in \RR, \|u\|\leq t \}.
\]
The closed convex set $\SOC_{n}$ is also called the ice-cream cone or the Lorentz cone. This problem, called \emph{second-order conic system feasibility},  arises in the second-order cone programming (SOCP)~\cite{Alizadeh:2003,Lobo:1998}, in which a linear function is minimized over the intersection of a polyhedral set and the intersection of second-order cones, and where an initial feasible point needs to be found~\cite{Cucker:2015}. 

In order to execute our tests, we randomly generate instances of the polyhedron $\Omega$, where $n$ is fixed as \num{200} and $m$ is a random value between $n/3$ and $n$. We guarantee that $\Sset$ is nonempty by the following procedure. We sample a nonzero point $u\in \RR^{n-1}$ from the standard normal distribution, assuring its  norm lies between \num{5} and \num{15},  and form the vectors $\hat z  \coloneqq (\norm{u},u)$, which  clearly belongs to $\SOC_n$,  and $\hat d  \coloneqq (-\norm{u},u)$. Note that $\hat z$ and $\hat d$ are orthogonal.
Next, we sample, from the standard normal distribution, unitary vectors  $a_i \in \RR^{n}$, $i=1,\ldots,m$,  such that $\langle {a_i}, {{\hat d}}\rangle  < 0$, \emph{i.e.}, the correspondent angles are strictly obtuse.  Each  $a_i$ is set as a row of matrix $A$, while $b\coloneqq A(\hat z - \tau \hat d)$, where $\tau\in [0,1]$. In this way,  $\Omega$ is a polyhedron and $\hat z \in \Omega$, guaranteeing the nonemptiness of $\Sset$, as required. Observe that the $a_i$'s are the generators of the polar cone of $\Omega$ at $\hat z - \tau \hat d$. By construction, if $\tau \in (0,1]$, the interior of $\Sset$ is nonempty  and, in particular, Assumption \ref{eb} is satisfied. If  $\tau \coloneqq 0$,  Assumption \ref{eb} may fail to be fulfilled.  In fact, we verified that whenever we ran an instance where $\tau$ was zero the error bound did not hold; this was checked by exploring the structure of the problem. 


Each of the instances we generate is run for \num{4} initial random points. Each initial point is also sampled from the standard normal distribution, with norms ranging from  \num{5} and \num{15},  and is accepted as long as it is not in $\Sset$.  In order to handle the projections onto the convex sets $\SOC_n$ and $\Omega$, we  employ in our implementation the open source \texttt{Julia} package \texttt{ProximalOperators.jl}~\cite{Stella:2022}.

\subsubsection{Instances where $\inte \Sset$ is nonempty}

We create \num{50} instances, summing up to \num{200} individual tests, taking $\tau \coloneqq \tfrac{1}{4}$; this value of $\tau$ assures that the interior of $\Sset$ is nonempty, so that the error bound condition \ref{eb} holds.
Let $(w^k)_{k\in \na}$ be any of the three sequences that we monitor, generated by cCRM, MAP, and CRMprod. We considered as tolerance $\varepsilon \coloneqq  10^{-6}$
and employed as  stopping criteria  the \emph{gap distance}, given by
\[\label{eq:gap_distance_PolySOC}
\|P_{\Omega}(w^k) - P_{\SOC_n}(w^k)\|< \varepsilon,
\]
which is reliable measure of infeasibility in view of \ref{eb} being satisfied. The projections computed to measure the gap distance can be utilized in the next iteration, thus this calculation does not add any extra cost.

The results displayed in   \Cref{fig:pp-PolySOC_tau=0.25}  and \Cref{tab:PolySOC_tau=0.25} clearly show a better performance of  cCRM over  MAP and CRMprod.   Recall that, cCRM uses four projections at each step;  thus, from \Cref{tab:PolySOC_tau=0.25} we can conclude the new proposed method  takes, in average, \num{4.95} iterations to converge. Of course, this was expected due to the result in \Cref{theorem:superlinearConv}, stating the superlinear convergence of cCRM.  

\begin{figure}[!ht]

  \centering
  \includegraphics[width=.7\textwidth]{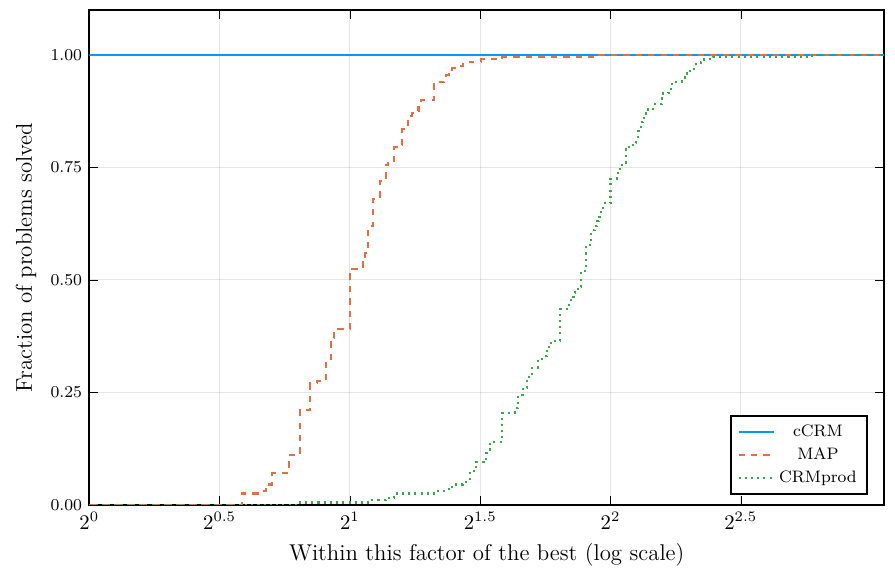}

\caption{Performance profile of experiments considering  a polyhedron and  a second order cone  with $\inte \Sset \neq \emptyset$.\label{fig:pp-PolySOC_tau=0.25}}
\end{figure}

\begin{table}[!ht]
  \caption{\label{tab:PolySOC_tau=0.25} Statistics of the experiments considering a polyhedron and a second order cones with $\inte \Sset \neq\emptyset$ (in number of projections).}
  \centering 
  \sisetup{
    table-parse-only,
    table-figures-decimal = 0,
    table-format = +3.1e-1,
    table-auto-round,
    separate-uncertainty
    %
    }
    \begin{tabular}{lSSSS}
    \toprule  
    & {\textbf{mean $\pm$ std}} &   {\textbf{median}}   &  {\textbf{min}} &   {\textbf{max}}     \\
    \cmidrule(lr){2-5}
    cCRM &      \num{ 19.8 +-   3.80}  &  20.0   &    8  &   28\\
    MAP &       \num{40.62 +-  10.20 }  & 39.0    &   24  &  72\\
    CRMprod &   \num{73.38 +-  21.94 }  & 70.0    &   28  &  140  \\
    \bottomrule
    \end{tabular}
  \end{table}

\subsubsection{Instances where $\inte \mathcal{S} = \emptyset$}

We report now the numerical experiments for instances in which  $\tau \coloneqq 0$, \emph{i.e.}, $\inte \mathcal{S}$ is empty; this is  a possible more challenging scenario because  \ref{eb} may be violated. We still use the gap distance \cref{eq:gap_distance_PolySOC} as a measure of infeasibility, even though   not as reliable as in the last experiments, due to the possible lack of \ref{eb}. Again, \num{50} instances are generated, so we gather the results of \num{200} tests.

 The performance profile of \Cref{fig:pp-PolySOC_tau=0.0} once again shows that cCRM is faster and more robust than its counterparts. In \Cref{tab:PolySOC_tau=0.0}, we can see that  cCRM took on average almost 3 times fewer projections than MAP while outperforming CRMprod, being more than 4 times faster, in number of projections required, to achieve the desired precision.

\begin{figure}[!ht]

    \centering
    \includegraphics[width=.7\textwidth]{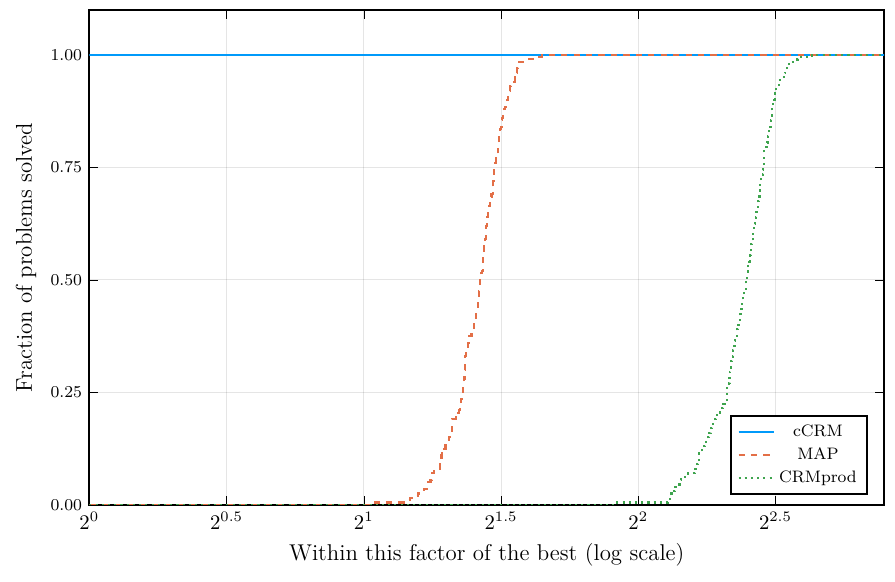}

\caption{Performance profile of experiments with polyhedral and second order cones feasibility with $\inte \Sset = \emptyset$.\label{fig:pp-PolySOC_tau=0.0}}
\end{figure}

\begin{table}[!ht]
\caption{\label{tab:PolySOC_tau=0.0} Statistics of the experiments with polyhedral and second order cones with $\inte \Sset = \emptyset$ (in number of projections).}
\centering 
\sisetup{
  table-parse-only,
  table-figures-decimal = 2,
  table-format = +3.1e-1,
  table-auto-round,
  separate-uncertainty
  %
  }
  \begin{tabular}{lSSSS}
  \toprule  
  & {\textbf{mean $\pm$ std}} &   {\textbf{median}}   &  {\textbf{min}} &   {\textbf{max}}     \\
  \cmidrule(lr){2-5}
  cCRM &     \num{75.14 +-32.82 }&    68.0   & 24  & 188  \\
  MAP &      \num{201.87+-91.78 }  &  185.0  & 62  & 490  \\
  CRMprod &  \num{396.23+-184.47  }  &  364.0  & 120 & 990  \\

  \bottomrule
  \end{tabular}
\end{table}

\section{Concluding remarks}\label{s7}

Circumcenter-type methods have been attracting substantial interest in the last few years. In the present work, we introduce and study the centralized circumcentered-reflection method (cCRM) for finding a point in the intersection of two closed convex sets. Global convergence of cCRM is established, as well as linear convergence under an error bound condition. Moreover,  superlinear convergence of cCRM is derived under  local smoothness of the boundaries of the sets and the assumption that their intersection has nonemtpy interior. We note that  cCRM does not employ  any  product space reformulation, which is a significant advance in the theory of generalized circumcenters. Our numerical tests are consistent with the theory we developed and reassure the Newtonian flavor of circumcenter schemes. An interesting topic for future research is the 
 development of multi-set centralization techniques.

  \begin{acknowledgements}
We thank the anonymous referees for their valuable suggestions which significantly improved  this manuscript. 
\end{acknowledgements}

\bibliographystyle{spmpsci}

\bibliography{refs}

\end{document}